\documentclass{amsart}

\usepackage{amsfonts}
\usepackage{latexsym,amssymb}
\usepackage{amsmath, amsbsy}
\usepackage{amsopn, amstext}
\usepackage{graphicx, color, epstopdf}
\usepackage{threeparttable}
\usepackage{float}
\usepackage{multirow}
\usepackage{subfigure}

\newtheorem{theorem}{Theorem}[section]
\newtheorem{lemma}[theorem]{Lemma}

\theoremstyle{definition}

\theoremstyle{remark}
\newtheorem{remark}[theorem]{Remark}

\numberwithin{equation}{section}

\newtheorem{proposition}{Proposition}

\numberwithin{equation}{section}
 \numberwithin{Lem}{section}
 \numberwithin{Defi}{section}
 \numberwithin{Theo}{section}
 \numberwithin{Rem}{section}
  \numberwithin{Coro}{section}
  \numberwithin{Fig}{section}
  
\def\NN{\hbox{\rlap{I}\kern.16em N}}
\def\NC{\hbox{\rlap{\kern.24em\raise.1ex\hbox
                  {\vrule height1.3ex width.9pt}}C}}

\def \bZ {\mathbb Z}

\begin{document}

\title[A $C^1$-conforming Petrov-Galerkin  method] 
{A $C^1$-conforming Petrov-Galerkin method for  convection-diffusion equations and superconvergence ananlysis over rectangular meshes}


\author{Waixiang Cao}
\address{School of Mathematical Science, Beijing Normal University, Beijing
l00875,  China.}
\email{caowx@bnu.edu.cn}
\thanks{The first author's research was supported in part by NSFC  grant No.11871106.}

\author{Lueling Jia}
\address{Beijing Computational Science Research Center,
Beijing,  100193, China.}
\email{lljia@csrc.ac.cn}

\author{Zhimin Zhang}
\address{Beijing Computational Science Research Center,
Beijing,  100193, China; and  Department of Mathematics, Wayne State University,
Detroit, MI 48202, USA.}
\email{zmzhang@csrc.ac.cn}
\thanks{The third author's research was supported in part by NSFC grants  No.11871092 and NSAF U1930402.}



\date{}

\dedicatory{}

\begin{abstract}
  In this paper, a new $C^1$-conforming Petrov-Galerkin method for convection-diffusion equations is designed and analyzed.
  The trail space of  the proposed method is a $C^1$-conforming ${\mathbb Q}_k$ (i.e., tensor product of polynomials of degree at most $k$) finite element space while the test space is taken as the $L^2$   (discontinuous) piecewise ${\mathbb Q}_{k-2}$ polynomial space. 
  Existence and uniqueness of the numerical solution is proved and optimal error estimates in all $L^2, H^1, H^2$-norms are established.
 In addition, superconvergence properties of the new method are investigated and superconvergence points/lines are identified at mesh nodes  (with order $2k-2$ for both function value and derivatives),  at roots of a special Jacobi polynomial, and at the Lobatto lines and  Gauss lines  with rigorous theoretical analysis.  
 In order to reduce the global regularity requirement, interior a priori error estimates in the $L^2, H^1, H^2$-norms are derived.   Numerical experiments are presented to confirm theoretical findings.
\end{abstract}

\maketitle

 \section{Introduction}

   The Petrov-Galerkin method, using different trail and test spaces, has been widely used in solving
   various partial differential equations such as  second-order wave equations \cite{pgwave}, electromagnetic problems \cite{pg-elec,pg-elec1},
   fluid mechanic equations \cite{pg4,pg6}, and so on.
    Classified by the continuity of the approximation space,  the existing  Petrov-Galerkin method can be roughly divided
   into three categories, i.e., the family of $C^0$ elements that require
the continuity of the numerical solution, the $L^2$ elements (was also called discontinuous Petrov-Galerkin method)
  whose trail space is not necessary to be continuous, and
 the $C^1$ elements that
require the continuity of the trail space and its first-order
derivatives. Comparing with  the $C^0$ and $L^2$ elements Petrov-Galerkin method (see e.g., \cite{dpg3,dpg1,dpg2}) or the counterpart
$C^1$ finite element method (see e.g., \cite{c1,c2,c3}), the $C^1$ Petrov-Galerkin method is still far from fully developed.

 The Petrov-Galerkin method we study in this paper is $C^1$-conforming, where we   use  $C^1$-conforming piecewise $\mathbb Q_k$ polynomials (the tensor product space)  as
 trial functions and  the $L^2$ discontinuous piecewise  $\mathbb Q_{k-2}$ polynomials as test functions.
 Comparing with the continuous Galerkin (i.e., $C^0$ element) and discontinuous Galerkin (i.e., $L^2$ element) methods,
 the most attractive feature  of  the proposed $C^1$-conforming method is the continuity of the derivative approximation across the element interface.
Note that the total  degrees of freedom  of  the $C^1$-conforming method is the same
   or less than the counterpart  $C^0$ and/or $L^2$ element methods over rectangular meshes with the same accuracy. In other words, the   $C^1$-conforming method  provides a better
    approximation  for derivatives (including the second-order derivatives) without increasing the computational cost. Furthermore, the discontinuous test space is used so  that the test functions
     can be locally computed on each element,
    which makes the  assembly of global matrices simpler than the counterpart $C^1$-conforming finite element method, where some test functions across several elements.

   The objective of the present study is to develop   a $C^1$-$L^2$ pair of  Petrov-Galerkin method (i.e., the trial space is $C^1$ while the test space is chosen as $L^2$), using the
   two-dimensional  convection-diffusion equations as model problems. We
    provide  a unified mathematical approach to establish
   convergence theory for the  proposed method including the
   optimal error estimates  in all $H^1, L^2, H^2$-norms and superconvergence results at some special points and lines.
   Note that superconvergence behavior has been investigated for many years. For an incomplete list of references, we refer to
  \cite{Babuska1996,Bramble.Schatz.math.com,Chen.C.M2012,ewing-lazarov-wang,Neittaanmaki1987,V.Thomee.math.comp,Wahlbin,Zhu.QD;LinQ1989}
  for $C^0$ finite element methods, and  \cite{Cai.Z1991,Cao;Zhang;Zou2012,Cao;zhang;zou:2kFVM,Chou_Ye2007,Xu.J.Zou.Q2009} for $C^0$ finite volume methods,
\cite{Adjerid;Massey2006,Adjerid;Weinhart2011,cao-shu-zhang-yang-nonlinear,Cao;zhang:supLDG2k+1,Cao;zhang;zou:2k+1,Chen;Shu:SIAM2010,Xie;Zhang2012,Yang;Shu:SIAM2012}
for discontinuous Galerkin methods, and \cite{zhang,zhang2008} for  spectral Galerkin methods.
Regardless of  rich literatures on the superconvergence study, the relevant work for $C^1$ element methods
is far from satisfied. Only very special and simple cases have been
discussed (see. e.g., \cite{Wahlbin,Bialeck,Bhal}).
To the best of our knowledge, no superconvergence analysis of   the $C^1$ Petrov-Galerkin method has been published yet until our recent work on $C^1$ Petrov-Galerkin and Gauss collocation methods for 1D two-point boundary value problems
in \cite{cao-jia-zhang}.

The  main superconvergence results  established in this paper include: 
1)   $h^{2k-2}$ superconvergence rate for approximations of both function value and the first-order derivatives at mesh nodes;  
2)  $h^{k+2}$ superconvergence rate  for the function value approximation at roots of a special Jacobi   polynomial;
3) $h^{k+1}$ superconvergence rate for the first-order and $h^k$ superconvergence rate for the second-order   derivative approximations at   Lobatto lines and Gauss lines, respectively;
4) as a by-product,  we also prove that  the Petrov-Galerkin solution is superconvergent towards a particular Jacobi projection of the exact solution in $H^2$, $H^1$, and $L^2$-norms.
By doing so, we present a full picture of superconvergence theory for the $C^1$ Petrov-Galerkin method, which gives us some   insights into the difference among the $C^0,C^1,L^2$ element methods.
We have found that the superconvergence points of the solution and its first-order derivative for the $C^1$ Petrov-Galerkin method  are different from those for the existing  $C^0$ Galerkin methods (e.g., FEM, FVM) and $L^2$ discontinuous Galerkin methods. 
The supreconvergence of the  second-order derivative approximation for the $C^1$ Petrov-Galerkin method is also novel.
Comparing with the $\mathbb Q_k$ $C^0$ Galerkin method (see, e.g., FEM in \cite{Chen.C.M2012}, FVM in \cite{Cao;zhang;zou:2kFVM})  for  the Poisson equation over rectangular meshes, which converges with  rate $h^{2k}$ at nodal points,  
the convergence rate at mesh nodes for the $C^1$ Petrov-Galerkin method   drops to  $h^{2k-2}$, while the convergence rate of the first-order derivative at mesh nodes
 improves from  $h^k$ to $h^{2k-2}$, which almost doubles the optimal convergence rate  $h^k$.

 The rest of the paper is organized as follows. In   Section 2, we present a $C^1$-$L^2$ Petrov-Galerkin method for two-dimensional
    convection-diffusion equations over rectangular meshes. In   Section  3, we prove the existence and uniqueness of the numerical
    scheme. In  Section  4, we construct a $C^1$-conforming Jacobi projection of the exact solution and
    study the approximation  and superconvergence properties of the special Jacobi projection.  Section 5 is the main and most
    technical part, where optimal error estimates in $L^2,H^1,H^2$-norms  and superconvergence behavior  at the mesh points (for solution and its first-order derivative approximations), at interior roots of Jacobi polynomials ( solution approximation), at
 Lobatto lines (the first-order derivative approximation) and Gauss lines (the second-order derivative approximation) are investigated. 
  In  Section  6, we establish some interior a priori error estimates in $H^2, H^1, L^2$-norms.
 Numerical experiments supporting our theory are presented in   Section  7. 
 Some concluding remarks are provided in   Section  8.

   Throughout this paper,  we adopt standard notations for Sobolev spaces such as $W^{m,p}(D)$ on sub-domain $D\subset\Omega$ equipped with
    the norm $\|\cdot\|_{m,p,D}$ and semi-norm $|\cdot|_{m,p,D}$. When $D=\Omega$, we omit the index $D$; and if $p=2$, we set
   $W^{m,p}(D)=H^m(D)$,
   $\|\cdot\|_{m,p,D}=\|\cdot\|_{m,D}$, and $|\cdot|_{m,p,D}=|\cdot|_{m,D}$. Notation $A\lesssim B$ implies that $A$ can be
  bounded by $B$ multiplied by a constant independent of the mesh size $h$.
  $A\sim B$ stands for $A\lesssim B$ and $B\lesssim A$.

\section{  A $C^1$ Petrov-Galerkin method}

   We consider the  following  convection-diffusion problem
\begin{eqnarray}\label{con_laws}
\begin{aligned}
     &- \nabla\cdot(\alpha \nabla u) +{\bf \beta}\cdot\nabla u + \gamma u=f, &\rm{in}\   \ \Omega=(a,b)\times(c,d),\\
     &u=0,  &\ \rm{on}\  \  \partial\Omega,
\end{aligned}
\end{eqnarray}
  where $\alpha\ge \alpha_0 > 0, \gamma-\frac{\nabla\cdot{\bf \beta}}{2}\ge 0, \gamma\ge 0$, $\alpha,{\bf \beta}=(\beta_1,\beta_2), \gamma\in L^{\infty}(\bar\Omega)$, and $f$ is real-valued function defined on $\bar\Omega$.
  Without loss of generality, we assume that $ \alpha,{\bf \beta},\gamma$ are all constants.
  The assumption is not essential since the analysis can be applied to that for variable coefficients as long as the above conditions are satisfied.

Let $a=x_{0}< x_{1}<\cdots<x_{M}=b$ and $c=y_{0}<
   y_{1}<\cdots<y_{N}=d$.
  For any positive integer $r$, we define
  $\bZ_{r}=\{1,2,\ldots, r\}$, and denote by
   ${\mathcal T}_h$ the rectangular partition of $\Omega$. That is,
\[
   {\mathcal T}_h=\{ \tau_{i,j}=[x_{i-1},
   x_{i}]\times[y_{j-1},y_{j}]:
   (i,j)\in\bZ_{M}\times\bZ_N\}.
\]
  For  any $\tau\in{\mathcal T}_h$, we denote by $h^x_\tau$, $h^y_\tau$
  the lengths of $x$- and $y$-directional edges of $\tau$, respectively.
$h$ is the maximal length
  of all edges, and $h_{\min}=\min_\tau (h^x_\tau,h^y_\tau).$
  We assume that the mesh ${\mathcal T}_h$ is {\it{quasi-uniform}} in the sense that there exists
  a constant $c$ such that $h\le c h_{\min}$. 

   We define the $C^1$ finite element space as follows:
\[
    V_h:=\{ v\in C^1(\Omega): \; v|_{\tau}\in \mathbb{Q}_k(x,y)=\mathbb {P}_k(x)\times \mathbb {P}_k(y), \tau\in{\mathcal T}_h\}^{},
\]
 where $\mathbb {P}_k$ denotes the space of
  polynomials of degree not more than $k$.
   Let
\[
    V_h^0:=\{ v\in V_h: \; v|_{\partial\Omega}=0\}.
^{}\]
  To design the Petrov-Galerkin method, we  define the test space $W_h$  as follows:
\begin{equation}\label{test:space}
    W_h:=\{ v\in L^2(\Omega): \; v|_{\tau}\in \mathbb{Q}_{k-2}(x,y)=\mathbb {P}_{k-2}(x)\times \mathbb {P}_{k-2}(y), \tau\in{\mathcal T}_h\}^{}.
\end{equation}
Then the $C^1$ Petrov-Galerkin method for solving \eqref{con_laws} is:  Find a $u_h\in V^0_h$ such that
 \begin{equation}\label{PG}
      a(u_h,v_h):= (- \nabla\cdot(\alpha \nabla u_h)+ {\bf\beta}\cdot\nabla u_h + \gamma u_h, v_h)=(f,v_h),\ \ \forall v_h \in W_h.
 \end{equation}
  Here $(u,v)=\sum_{\tau\in {\mathcal T}_h}\int_{\tau} (uv)(x,y)dxdy.$

   We would like to point out that the  method we proposed here is only one the several ways to define a  $C^1$  Petrov-Galerkin method. Actually, different choices of the test space $W_h$
   may  lead to  different numerical schemes. For example, other than the $L^2$ test space, we can also choose the $C^0$ space as our test space,
   i.e., $W_h\subset C^0(\Omega)$ is some subspace of the continuous finite element space. Throughout this paper, we focus our analysis on
   the $L^2$ test space, i.e., $W_h\subset L^2(\Omega)$ is defined by \eqref{test:space}.

 \section{Existence and uniqueness}
    In this section,  we discuss the existence and   uniqueness of the $C^1$ Petrov-Galerkin method \eqref{PG}. We begin with some estimates  for the bilinear form $a(\cdot,\cdot)$ of the Petrov-Galerkin method, which plays important role
   in our later analysis.
\begin{lemma}\label{lemma-1}
  Given any $v\in V_h^0$,  suppose that $\varphi\in H^2(\Omega)$ is the solution the following dual problem:
\begin{eqnarray}\label{dual:problem}
   -\nabla\cdot(\alpha \nabla \varphi)-\beta\cdot\nabla \varphi+\gamma \varphi=v\ \ \rm{in}\   \ \Omega,\
     {\rm and}\  \varphi=0,  \ \rm{on}\  \  \partial\Omega.
\end{eqnarray}
  Denote by ${\mathcal I}_h\varphi \in \mathbb Q_1 \subseteqq W_h$ the bi-linear interpolation function of $\varphi$.
  Then
\begin{eqnarray}\label{eqq:1}
   &&\|v\|_0^2\lesssim h^4 (\|v_{xxy}\|^2_0+\|v_{xyy}\|_0^2)+|a(v,{\mathcal I}_h\varphi)|,\\\label{eqq:2}
  && \|\triangle v\|^2_0+\|v_{xxy}\|^2_0+\|v_{xyy}\|_0^2\lesssim |a(v,v_{xxyy})|+\|v\|_0^2.
\end{eqnarray}
\end{lemma}
\begin{proof}
  First, from   the dual problem \eqref{dual:problem} and the integration by parts, we have
\begin{eqnarray*}
  \|v\|_0^2&=&(v,-(\nabla\cdot(\alpha \nabla \varphi)-\beta\cdot\nabla \varphi+\gamma \varphi)
     =(-\nabla\cdot(\alpha \nabla v)+\beta\cdot\nabla v+\gamma v, \varphi)\\
     &=&(-\nabla\cdot(\alpha \nabla v)+\beta\cdot\nabla v+\gamma v, \varphi-{\mathcal I}_h\varphi+{\mathcal I}_h\varphi)\\
     &\lesssim& h^2(\|\triangle v\|_0+\|v\|_1)\|\varphi\|_2+|a(v,{\mathcal I}_h\varphi)|\lesssim h^2(\|\triangle v\|_0+h^{-1}\|v\|_0)\|v\|_0+|a(v,{\mathcal I}_h\varphi)|.
\end{eqnarray*}
 Here in the last step,
  we have used  the $H^2$ regularity $\|\varphi\|_2\lesssim \|v\|_0$ and
  the inverse inequality
\[
   \|v\|_1\lesssim h^{-1}\|v\|_0,\ \ \forall v\in V_h.
\]
 Consequently, if $h$ is sufficiently small,  then
\begin{equation}\label{eqq:4}
   \|v\|_0^2\lesssim h^4\|\triangle v\|^2_0+|a(v,{\mathcal I}_h\varphi)|.
\end{equation}
  On the other hand,
  noticing that for any function $v\in V^0_h$,
$\partial_x^i v, i\ge 1$  is continuous about $y$ satisfying
$\partial_x^i v (x,c)=\partial_x^iv_{}(x,d)=0$. Similarly,
$\partial_y^i v, i\ge 1$  is a continuous function about $x$ satisfying
 $\partial_y^iv(a,y)=\partial_y^iv(b,y)$.  Then
\[
   v_{xx}(x,y)=\int_{c}^y v_{xxy}(x,y) dy,\ \ v_{yy}(x,y)=\int_{a}^xv_{yyx}(x,y) dx.
\]
   By Poinc{a}r\'e inequality,
\[
    \| v_{xx}\|_0+\|v_{yy}\|_0\lesssim \|v_{xxy}\|_0+\|v_{yyx}\|_0.
\]
    Therefore,
 \[
     \|\triangle v\|_0\le \|v_{xx}\|_0+\|v_{yy}\|_0
     \lesssim \|v_{xxy}\|_0+\|v_{yyx}\|_0,
 \]
   which yields (together with \eqref{eqq:4}) the desired result \eqref{eqq:1}.

 We next consider \eqref{eqq:2}.  By integration by parts, the inverse inequality and \eqref{eqq:1}, there holds
 for any positive constant $\epsilon$,
 \begin{eqnarray*}
   (v_{xy},v_{xy})=-(v_x,v_{xyy})
   &\le &  \frac{1}{4\epsilon}\|v\|_1^2+\epsilon\|v_{xyy}\|_0^2\\
   &\le& C \|v\|_0^2+ {\epsilon} (\|v_{xyy}\|^2_0+\|v_{xxy}\|^2_0+\|\triangle v\|_0^2).
\end{eqnarray*}
  Consequently,
\begin{eqnarray*}\label{eqq:5}
\begin{split}
   |({\bf \beta}\cdot\nabla v,v_{xxyy})|&=|( v_{xy}, { \beta_1}v_{xxy}+{ \beta_2}v_{yyx})|
   \le \frac{c_0}{\alpha}\|v_{xy}\|_0^2+\frac{\alpha}{4}(\|v_{xxy}\|^2+\|v_{yyx}\|_0^2) &\\
   &\le  (\frac{c_0\epsilon}{\alpha}+\frac{\alpha}{4}) (\|v_{xyy}\|^2_0+\|v_{xxy}\|^2_0)+C\|v\|_0^2+\frac{c_0\epsilon}{\alpha}\|\triangle v \|_0^2,&
\end{split}
\end{eqnarray*}
where $c_0=\max(\beta_1^2,\beta_2^2)$.
  Recalling the definition of $a(\cdot,\cdot)$ and  using the integration by parts again, we derive that
\begin{align*}\label{eq:000}
   a(v,v_{xxyy})&=\alpha(\|v_{xxy}\|_0^2+\|v_{xyy}\|_0^2)+\gamma \|v_{xy}\|_0^2+({\bf \beta}\cdot\nabla v,v_{xxyy})\\\nonumber
       &\ge (\frac{3\alpha}{4}-\frac{c_0\epsilon}{\alpha}) (\|v_{xxy}\|_0^2+\|v_{yyx}\|_0^2)-\frac{c_0\epsilon}{\alpha}\|\triangle v \|_0^2+\gamma \|v_{xy}\|^2_0-C\|v\|_0^2\\\nonumber
       &\ge  (\frac{3\alpha}{4}-C_0\epsilon) (\|v_{xxy}\|_0^2+\|v_{yyx}\|_0^2)+\gamma \|v_{xy}\|^2_0-C\|v\|_0^2.
\end{align*}

   By choosing a small $\epsilon$,  we obtain \eqref{eqq:2}  directly.
  This finishes our proof. $\Box$
\end{proof}

  Now we are ready to prove the existence and uniqueness results for the $C^1$ Petrov-Galerkin method.
\begin{theorem}\label{theo:4}
    The  $C^1$ Petrov-Galerkin method \eqref{PG}
    has one and only one solution, provided that the mesh size is sufficiently small.
\end{theorem}
\begin{proof}
   We shall prove that the homogeneous problem  has a
unique zero solution.  To this end, we assume that $f=0$ and prove the numerical scheme \eqref{PG}
admits a solution $u_h=0$.

  Noticing that $ \partial^4_{xxyy}u_h \in W_h, {\mathcal I}_h\varphi\in W_h$ for any function $\varphi$,  we have
\[
   a(u_h, \partial^4_{xxyy} u_h)=0,\ \ a(u_h, {\mathcal I}_h\varphi)=0.
\]
 Then from \eqref{eqq:1}-\eqref{eqq:2}, we have $\|u_h\|_0=0$ and thus
\[
    u_h\equiv 0.
\]
  This finishes the proof. $\Box$

\end{proof}

  \section{A specially constructed Jacobi projection}

     In this section, we define a $C^1$ Jacobi projection of the exact solution and study  its
      approximation property, which is essential for the establishment of the
     superconvergence results for the numerical solution $u_h$, especially the discovery of superconvergence points.

   We begin with some preliminaries. First,  we  introduce  the four Hermite interponant basis functions on the interval $[-1,1]$, which are given by
\begin{eqnarray*}
  && \psi_{-1}(s)=\frac{1}{4}(s+2)(1-s)^2,\ \ \psi_{1}(s)=\frac{1}{4}(2-s)(1+s)^2,\\
 && \chi_{-1}(s)=\frac{1}{4}(s+1)(1-s)^2,\ \ \chi_{1}(s)=\frac{1}{4}(s-1)(1+s)^2.
\end{eqnarray*}

  Second, we denote by $J_n^{r,l}(s),\  r,l>-1$,
   the standard Jacobi polynomials of degree $n$ over $(-1,1)$, which
     are orthogonal with respect to the Jacobi weight function
     $\omega_{r,l}(s):=(1-s)^{r}(1+s)^{l}$. That is,
 \[
     \int_{-1}^1J_n^{r,l}(s)J_m^{r,l}(s)\omega_{r,l}(s)ds
     =\kappa_n^{r,l}\delta_{mn},
 \]
 where $\delta$ denotes the Kronecker symbol and $ \kappa_n^{r,l}=\|J_n^{r,l}\|^2_{\omega_{r,l}}$. 
Note that when $r=l=0$, the Jacobi polynomial $J_n^{r,l}$ is reduced to the standard Legendre polynomial. That is
 $J_{n}^{0,0}(s)=L_{n}(s)$ with
  $L_{n}(s)$  being the  Legendre polynomial of degree $n$ over $[-1,1]$.
 We extend the definition of the classical Jacobi polynomials to the case where
     both parameters $r,l \le -1$
 \begin{eqnarray}\label{Jacobi:2}
     J_n^{r,l}(s):=(1-s)^{-r}(1+s)^{-l}J_{n+r+l}^{-r,-l}(s),\ \ r,l\le -1.
 \end{eqnarray}
  By taking $r=l=-2$ in \eqref{Jacobi:2}, we get a sequence of Jacobi polynomials $\{J_n^{-2,-2}\}_{n=4}^{\infty}$ with
\begin{eqnarray}\label{Jacobi:1}
     J_n^{-2,-2}(s):=(1-s)^{2}(1+s)^{2}J_{n-4}^{2,2}(s), \ \ \forall n\ge 4.
 \end{eqnarray}
  Apparently,  there holds
 \begin{equation}\label{eq:3}
     J_n^{-2,-2}(\pm 1)=0,\ \ \partial_sJ_n^{-2,-2}(\pm 1)=0.
 \end{equation}
   Denoting
\[
   J_0^{-2,-2}(s)=\psi_{-1}(s),\ \ J_1^{-2,-2}(s)=\psi_{1}(s),\ \ J_2^{-2,-2}(s)=\chi_{-1}(s),\ \ J_3^{-2,-2}(s)=\chi_{1}(s),
\]
  then $\{J_n^{-2,-2}\}_{n=0}^{\infty}$ constitutes the basis function of $C^1$ over $[-1,1]$. We also refer to \cite{shen2011} for more detailed information and 
  discussions about the Jacobi polynomials. 

 Third, we denote by $\phi_{n+1}$ for $n\ge 1$  the Lobatto polynomial of degree $n+1$ over $[-1,1]$, which is defined by
\begin{equation}\label{lob}
   \phi_{n+1}(s)=\int_{-1}^s L_n(s) ds=\frac{1}{2n+1}(L_{n+1}-L_{n-1})=\frac{1}{n(n+1)}(s^2-1)L'_{n}(s).
\end{equation}
    The above
   Jacobi and Lobatto polynomials will be frequently used  in our  later superconvergence analysis.

   Now we are ready to present the truncated Jacobi projection.
   Given any function $u\in C^1(\Omega)$, suppose $u(x,y)$ has the following Jacobi expansion in each element $\tau_{ij},(i,j)\in\bZ_M\times\bZ_N$
  \begin{eqnarray}\label{uu}
      u(x,y)|_{\tau_{ij}}=\sum_{p=0}^{\infty}\sum_{q=0}^{\infty} u_{pq} J_{i,p}^{-2,-2}(x) J_{j,q}^{-2,-2}(y),
  \end{eqnarray}
    where $ J_{i,p}^{-2,-2}(x)=J_p^{-2,-2}(\frac{2x-x_i-x_{i-1}}{h_i})=J_p^{-2,-2}(s),\ \ s\in [-1,1]$, is the Jacobi polynomial of degree $p$ over $(x_{i-1},x_i)$, and $u_{pq}$ are some coefficients to be determined.
    Then the truncated Jacobi projection $u_I\in V_h$ of $u$ is defined as follows:
\begin{equation}\label{eq: interpolation}
    u_I(x,y)|_{\tau_{ij}}:=\sum\limits_{p=0}^{k}\sum\limits_{q=0}^{k} u_{pq}   J_{i,p}^{-2,-2}(x) J_{j,q}^{-2,-2}(y).
\end{equation}
  Note that when $k=3$, the truncated Jacobi projection $u_I$ is reduced to the Hermite interpolation of $u$.

  For  all $\tau=\tau_{ij}$, a direct calculation yields
 \begin{align}\label{error:decomposition}
    (u-u_I)|_{\tau}=\sum_{p=k+1}^{\infty}\sum_{q=k+1}^{\infty}u_{pq}J_{i,p}^{-2,-2}(x) J_{j,q}^{-2,-2}(y)
    =(E^xu+E^yu-E^xE^yu),
 \end{align}
    where 
\begin{eqnarray}\label{Ex}
  && E^xu(x,y)|_{\tau_{ij}}=\sum_{p=k+1}^{\infty}\sum_{q=0}^{\infty}u_{pq} J_{i,p}^{-2,-2}(x)  J_{j,q}^{-2,-2}(y),\\
   && E^yu(x,y)|_{\tau_{ij}}=\sum_{p=0}^{\infty}\sum_{q=k+1}^{\infty}u_{pq}J_{i,p}^{-2,-2}(x)  J_{j,q}^{-2,-2}(y),\\\label{Ey}
   && E^xE^yu(x,y)|_{\tau_{ij}}=\sum_{p=k+1}^{\infty}\sum_{q=k+1}^{\infty}u_{pq}J_{i,p}^{-2,-2}(x)  J_{j,q}^{-2,-2}(y).
\end{eqnarray}

Note that $E^xu$ is actually the one dimensional residual functions  along the $x$-direction while the other variable $y$ is fixed. 
Similar for $E^yu, E^xE^yu$.

We have the following properties for the residual functions $E^xu, E^yu$ and $E^xE^yu$.
\begin{lemma}\label{lemma:1}
 Assume that $u\in W^{l,\infty}(\Omega)\cap C^1(\Omega)$  with $0<l\le k+1$,  and $u_I$ is the truncated Jacobi projection of $u$ defined by \eqref{eq: interpolation}.
 Let $u-u_I=E^xu+E^yu-E^xE^yu$ with $E^xu,E^yu, E^xE^yu$ given by \eqref{Ex}-\eqref{Ey}.
 There holds for $m=\infty, 2$ and $p=0,1$ the following results:
\begin{enumerate}
 \item [1.]The function $E^xu(\cdot,y)\in C^1(\cdot, y)$ is continuous about $y$ and
 \begin{eqnarray*}\label{eq:4}
  &&\partial_x^pE^xu(x_i,y)=0,\ \partial_{xx}^2E^xu\bot {\mathbb P}_{k-2}(x),\  \partial_y^nE^xu(x,y)=E^x(\partial_y^nu), \forall n,\\\label{eq:5}
  &&\|E^xu\|_{0,m}+h\|\partial_xE^xu\|_{0,m}+h^2\|\partial^2_{xx}E^xu\|_{0,m}\lesssim h^l\|u\|_{l,m}.
\end{eqnarray*}
 \item [2.] The function $E^yu(x,\cdot)\in C^1(x,\cdot)$ is continuous about $x$ and
\begin{eqnarray*}\label{eq:6}
  &&\partial_y^pE^yu(x,y_j)=0,\ \partial_{yy}^2E^yu\bot {\mathbb P}_{k-2}(y),\ \partial_x^nE^yu(x,y)=E^y(\partial_x^nu),\ \forall n,\\\label{eq:7}
  && \|E^yu\|_{0,m}+h\|\partial_yE^yu\|_{0,m}+h^2\|\partial_{yy}^2E^yu\|_{0,m}\lesssim h^l\|u\|_{l,m}.
\end{eqnarray*}
\item[3.]  The function $E^xE^yu$ is of high order, i.e., there holds for $r,l \le k+1$, 
\begin{align*}\label{eq:8}
   \|E^xE^yu\|_{0,m}+h\|E^xE^yu\|_{1,m}+h^2\|E^xE^yu\|_{2,m}\lesssim h^{r+l}\|u\|_{r+l,m}. 
\end{align*}
\end{enumerate}
\end{lemma}
  Here we omit the proof  and refer to \cite{cao-jia-zhang} for more detailed discussions  about the one dimensional truncated Jacobi projection.

We next study the approximation and superconvergence propeties of $u_I$. To this end, we first introduce some special points and lines on the whole domain and then prove that $u_I$ is superconvergent at this special points and lines.

Let $R_{p}, p\in \bZ_{k-3}$ be
   the $k-3$  zeros of $J_{k+1}^{-2,-2}(s)$ except the point $s=-1,1$, and
   $l_{p}, p\in\bZ_{k}$ the  $k$
  Gauss-Lobatto points, i.e., $l_{p}, p\in\bZ_{k}$ are zeros of $\partial_sJ_{k+1}^{-2,-2}(s)$,
   and $G_p, p\in\bZ_{k-1}$ the Gauss points of degree $k-1$ in $[-1,1]$ (i.e., the zeros of $L_{k-1}$), respectively.
  Then for all $\tau=\tau_{i,j}\in {\mathcal T}_h, (i,j)\in\bZ_M\times\bZ_N$, denote
  $h_i^x=x_i-x_{i-1}, h_j^y=y_j-y_{j-1}$, and
\[
   {\mathcal R}_{\tau}=\{P: P=(R_{p}^{\tau,x},R_{q}^{\tau,y}), p,q\in\bZ_{k-3}\},\ \ {\mathcal R}=\bigcup_{\tau\in {\mathcal T}_h}{\mathcal R}_{\tau},
\]
   where
\[
  R_{p}^{\tau,x}=\frac{1}{2}(x_{i-\frac 12}+x_{i+\frac
  12}+h_i^xR_{p}),\ \  R_{p}^{\tau,y}=\frac{1}{2}(y_{j-\frac 12}+y_{j+\frac
  12}+h_j^yR_{p}).
\]
  Denote by
\begin{eqnarray*}
  &&{\mathcal E}_x^l=\{(x,y): x=l_{p}^{\tau,x}, y\in [c,d], {p\in\bZ_{k}},\tau\in
  {\mathcal T}_h\},\\
  &&{\mathcal E}_y^l=\{(x,y): y=l_{p}^{\tau,y}, x\in [a,b], {p\in\bZ_{k}},\tau\in
  {\mathcal T}_h\}
\end{eqnarray*}
  the set of vertical and horizontal edges of all interior Lobatto points
along the $x$-direction and the $y$-direction, respectively. Here
\[ l_{p}^{\tau,x}=\frac{1}{2}(x_{i-\frac 12}+x_{i+\frac
  12}+h_i^xl_{p}),\ \  l_{p}^{\tau,y}=\frac{1}{2}(y_{j-\frac 12}+y_{j+\frac
  12}+h_j^yl_{p}),\ \ \forall\tau=\tau_{ij}.
\]
Similarly,  the interior Gauss lines along the $x,y$-direction  are defined  as
\begin{eqnarray*}
  &&{\mathcal E}_x^g=\{(x,y): x=g_{p}^{\tau,x}, y\in [c,d], p\in\bZ_{k-1},\tau\in
  {\mathcal T}_h\},\\
  &&{\mathcal E}_y^g=\{(x,y): y=g_{p}^{\tau,y}, x\in [a,b], p\in\bZ_{k-1},\tau\in
  {\mathcal T}_h\},
\end{eqnarray*}
where $(g_{p}^{\tau,x},g_{q}^{\tau,y}), (p,q)\in (k-1)\times(k-1)$ denotes the $(k-1)^2$ Gauss points in $\tau$.

  We have the following approximation properties for the Jacobi projection $u_I$.
 \begin{proposition}\label{proposition:1}
 Assume that $u\in W^{l,\infty}(\Omega), l\ge 2$ is the solution of \eqref{con_laws}, and $u_I$ is the truncated Jacobi  projection of $u$ defined by \eqref{eq: interpolation}. The following orthogonality and approximation properties hold true.
  \begin{enumerate}
   \item [1.] Optimal error estimates:
 \begin{eqnarray*}\label{optimaluI}
    \|u-u_I\|_{0,m}\lesssim h^{r}\|u\|_{r,m},\ \ r\le \min(k+1,l), \ m=2,\infty. 
 \end{eqnarray*}
 \item [2.]  Exactly the same at the mesh nodes for both the function value and the first-order derivative:
\begin{eqnarray*}
  (u-u_I)(x_i,y_j)=0,\ \  \nabla(u-u_I)(x_i,y_j)=0. 
\end{eqnarray*}
  \item [3.]  Superconvergence of function value approximation on  roots of  $J_{i,k+1}^{-2,-2}(x)J_{j,k+1}^{-2,-2}(y)$:
  \begin{eqnarray*}\label{approx:super12}
  |(u-u_I) (P) | \lesssim   h^{r}|u|_{r,\infty},\ \ \forall P\in{\mathcal R},\  r\le \min(k+2,l). 
\end {eqnarray*}
 \item [4.] 
 Superconvergence of the first-order derivative on  Gauss-Lobatto lines, i.e., for $r\le \min(k+2,l)$, 
 \begin{eqnarray*}\label{approx:super3}
 |\partial_x(u-u_I) (P_1)| + |\partial_y(u-u_I) (Q_1)| \lesssim h^{r-1}|u|_{r,\infty}, 
\end{eqnarray*}
where $P_1\in {\mathcal E}_x^l, Q_1\in {\mathcal E}_y^l$ denotes the Lobatto lines along the $x$ and $y$ directions, respectively.
\item [5.] Superconvergence of  the second-order derivative on Gauss lines and Lobatto points, i.e., there holds for $r\le \min(k+2,l)$, 
\begin{eqnarray*}\label{approx:super4}
   |\partial_{xx}^2(u-u_I)(P_2)|+|\partial_{yy}^2(u-u_I)(Q_2)| +|\partial_{xy}^2(u-u_I)(P_3)| \lesssim h^{r-2}|u|_{r,\infty}, 
\end{eqnarray*}
  where $P_2\in  {\mathcal E}_x^g, Q_2\in {\mathcal E}_y^g $ denotes the Gauss lines along the $x$ and $y$ direction, respectively,
  and $P_3\in{\mathcal L}$ with ${\mathcal L}$ the set of  Lobatto points on the whole domain, i.e., ${\mathcal L}=\{(l_{p}^{\tau,x},l_{q}^{\tau,y}): \tau\in{\mathcal T}_h, (p,q)\in k\times k\}$.
\end{enumerate}
 \end{proposition}
 \begin{proof}  For any fixed $y$, there holds  for any $r\le \min(k+2,l)$ (see   \cite{cao-jia-zhang} )
 \[
    |E^xu(R_{m}^{\tau,x}, y)|+h|\partial_xE^xu(l_{n}^{\tau,x},y)| +h^2|\partial_{xx}E^xu(g_{p}^{\tau,x},y)| \lesssim h^{r}\|u\|_{r,\infty}. 
 \]
  Following the same arguments, we have 
 \[
    |E^yu(x,R_{m}^{\tau,y})|+h|\partial_yE^yu(x,l_{n}^{\tau,y})| +h^2|\partial_{yy}E^yu(x,g_{p}^{\tau,y})| \lesssim h^{r}\|u\|_{r,\infty}. 
 \]  
   Then the desired results follow from the error equation \eqref{error:decomposition} and 
    the estimates of $E^xu, E^yu, E^xE^yu$ in Lemma \ref{lemma:1}.  $\Box$


  %
 \end{proof}

    As we may observe, if we choose $l=k+1$, the point-wise error estimates in the above Proposition 
    indicate  s  superconvergent phenomenon 
of $u_I$ at mesh nodes,  at roots of  the Jacobi polynomial,  and at the Lobatto lines and  Gauss lines.

  \section{Error estimates and  superconvergence analysis}

  In this section, we present error estimates and  study superconvergence properties of the $C^1$ Petrov-Galerkin method  for \eqref{con_laws}.  In the rest of this paper,  we use the following notations:
\begin{equation}
   e=u-u_h,\ \ \xi=u_I-u_h,\ \ \eta=u-u_I.
\end{equation}

\subsection{Optimal error estimates}

\begin{theorem}\label{theo:1}
   Assume that $u\in H^{l}(\Omega)$ is the solution of \eqref{con_laws}, and $u_h$ is the solution of \eqref{PG}. Then
\begin{equation}\label{optimal:1}
   \|u-u_h\|_0+h\|u-u_h\|_1+h^2\|u-u_h\|_2\lesssim h^{r}\|u\|_{r+1},\ \ r\le \min(l-1,k+1). 
\end{equation}
\end{theorem}
\begin{proof}
    In light of \eqref{eqq:2} and the orthogonality $a(e,w)=a(\xi+\eta,w)=0$ for all $w\in W_h$, we have
\begin{eqnarray*}
   \|\triangle\xi\|^2_0+\|\xi_{xxy}\|_0^2+\|\xi_{xyy}\|_0^2&\lesssim & |a(\xi,\xi_{xxyy})|+\|\xi\|_0^2\\
   &\lesssim & |a(\eta,\xi_{xxyy})|+|a(\eta,{\mathcal I}_h\varphi)|,
\end{eqnarray*}
  where  $\varphi$ is the solution of the problem \eqref{dual:problem} with $v=\xi$.
   By  \eqref{error:decomposition}, we have
\begin{eqnarray*}
   a(\eta,w)=a(E^xu,w)+a(E^yu,w)-a(E^xE^yu,w),\ \ \forall w\in W_h.
\end{eqnarray*}
  Recalling the definition of the bilinear form and using and the integration by parts and the properties of $E^xu$ in Lemma \ref{lemma:1}, we derive  for all $\mu \le \min(k+1,l-2)$ 
\begin{eqnarray*}
   |a(E^xu,\xi_{xxyy})| 
         &=&|(-\alpha E^xu_{yy}+\beta_1\partial_xE^xu+\beta_2E^xu_y+\gamma E^xu,\xi_{xxyy})|\\
         &=&|(\partial_x(\alpha E^xu_{yy}-\beta_2 \partial_yE^xu-\gamma E^xu),\xi_{xyy})|+|\beta_1(\partial_{x}\partial_yE^xu, \xi_{xxy})|\\
         &\lesssim &
         h^{\mu-1}\|u\|_{\mu+2}(\|\xi_{xxy}\|_0+\|\xi_{xyy}\|_0). 
\end{eqnarray*}
  Consequently, 
 \[
    |a(E^xE^yu,\xi_{xxyy})| \lesssim  h^{\mu-1}\|E^yu\|_{\mu+2}(\|\xi_{xxy}\|_0+\|\xi_{xyy}\|_0)\lesssim h^{\mu-1}\|u\|_{\mu+2}(\|\xi_{xxy}\|_0+\|\xi_{xyy}\|_0). 
 \]
  By the same arguments, there holds
 \[
     |a(E^yu,\xi_{xxyy})|\lesssim h^{\mu-1}\|u\|_{\mu+2}(\|\xi_{xxy}\|_0+\|\xi_{xyy}\|_0).
 \]
   Then 
\begin{equation}\label{e-10}
  |a(\eta,\xi_{xxyy})|\lesssim h^{\mu-1}\|u\|_{\mu+2}(\|\xi_{xyy}\|_0+\|\xi_{xxy}\|_0),\ \  \mu \le \min(k+1,l-2). 
\end{equation}

   Now we consider the term $a(\eta,{\mathcal I}_h\varphi)$. 
  Noticing that $\varphi=0$ on $\partial\Omega$, we have from the integration by parts
 \begin{eqnarray*}
   |a(\eta,{\mathcal I}_h\varphi)|&=&|a(\eta,{\mathcal I}_h\varphi-\varphi)|+ |a(\eta,\varphi)|\\
                           &=&|a(\eta,{\mathcal I}_h\varphi-\varphi)|+|(\eta, -\alpha\triangle \varphi -{\bf \beta}\cdot\nabla \varphi +\gamma\varphi)|\\
                           &\lesssim & (h^{2}\|\eta\|_{2} +\|\eta\|_0)\|\varphi\|_{2}\lesssim h^{\mu'}\|u\|_{\mu'}\|\xi\|_{0},
\end{eqnarray*}
  where $\mu'\le \min(l,k+1)$, and 
   in the last step, we have used the $H^2$ regularity $\|\varphi\|_2\lesssim \|\xi\|_0$.
  Then we choose $v=\xi$ in \eqref{eqq:1}-\eqref{eqq:2} to obtain
 \begin{eqnarray*}
    && \|\xi\|_0^2\lesssim h^{2\mu'}\|u\|^2_{\mu'}+h^4(\|\xi_{xxy}\|_0^2+\|\xi_{xyy}\|_0^2),\\
    &&\|\triangle\xi\|^2_0+\|\xi_{xyy}\|^2_0+\|\xi_{xxy}\|^2_0\lesssim h^{2(\mu-1)}\|u\|^2_{\mu+2}+\|\xi\|^2_0.
 \end{eqnarray*}
  Consequently,  there holds for $ \mu \le \min(k+1,l-2), \mu' \le \min(l-1,k+1)$, 
\begin{equation}\label{e-17}
   \|\triangle\xi\|_0+\|\xi_{xyy}\|_0+\|\xi_{xxy}\|_0\lesssim h^{\mu-1}\|u\|_{\mu+2},\ \ 
   \|\xi\|_0\lesssim h^{\mu'}\|u\|_{\mu'+1}. 
\end{equation}
  As for the $H^1$-norm error estimate,  a direct calculation from the integration by parts  yields 
\[
   (\xi_x,\xi_x)+(\xi_y,\xi_y)=-(\xi,\xi_{xx})-(\xi,\xi_{yy})\lesssim \|\xi\|_0|\xi|_2\lesssim h^{2(\mu'-1)}\|u\|^2_{\mu'+1}. 
\]
  Then the desired result \eqref{optimal:1} follows from the triangle inequality and approximation properties of $u_I$.
  The proof is complete. $\Box$
\end{proof}

\subsection{Superconvergence analysis}

  In this subsection, we study superconvergence properties of the $C^1$ Petrov-Galerkin methods. 
  As the superconvergence analysis would require more strong regularity assumption on the smoothness of $u$ than one would need to obtain the counterpart optimal convergence rate, 
  we suppose the exact solution  $u$ is smooth enough in our superconvergence analysis.  In our later section, we discuss the interior estimates, i.e., the error in an interior domain $\Omega$, 
  with less requirements on the smoothness of $u$ on the whole domain $\Omega$.

 \begin{theorem}\label{theo:02}
 Assume that $u\in H^{k+3}(\Omega)$ is the solution of \eqref{con_laws}, and $u_h$ is the solution of \eqref{PG}.
The following superconvergence properties hold true.
\begin{enumerate}
\item [1.] Supercloseness result between $u_h$ and $u_I$ in all $H^2, H^1, L^2$-norms:
\begin{equation}\label{e-18}
   \|u_h-u_I\|_{1}+h\|u_h-u_I\|_2\lesssim h^{k+1}\|u\|_{k+3},\ \ \|u_h-u_I\|_0\lesssim h^{\min {(k+2,2k-2)}}\|u\|_{k+3}.
\end{equation}
  \item [2.] Superconvergence of  the function value  on  roots of  $J_{i,k+1}^{-2,-2}(x)J_{j,k+1}^{-2,-2}(y)$ in average sense for $k\ge 4$, i.e., 
 \begin{eqnarray}\label{super:pg1}
  e_{u,J}:= \Big(\frac {1}{NM}\sum_{P\in{\mathcal R}}\big(u-u_h\big)^2(P)\Big)^{\frac 12}\lesssim h^{k+2}\|u\|_{k+3}.
\end{eqnarray}
\item [3.] Superconvergence of the first-order derivative on Lobatto  lines in average sense, i.e.,
\begin{equation}\label{super:k3}
   e_{\nabla u,l}:=\Big(\frac {1}{N_x}\sum_{P_i\in {\mathcal E}_x^l}\partial_x\big(u-u_h\big)^2\big(P_i\big)+\frac {1}{N_y}\sum_{Q_i\in {\mathcal E}_y^l }\partial_y\big(u-u_h\big)^2\big(Q_i\big)\Big)^{\frac 12}\lesssim h^{k+1}\|u\|_{k+3}.
\end{equation}
\item [4.] Superconvergence of  the second-order derivative on   Gauss line in average sense. That is, 
\begin{equation}\label{super:k4}
   e_{\triangle u,g}:= \Big(\frac {1}{M_x}\sum_{P_i\in {\mathcal E}_x^g}\partial_{xx}^2\big(u-u_h\big)^2(P_i)+\frac {1}{M_y}\sum_{Q_i\in {\mathcal E}_y^g }\partial_{yy}^2(u-u_h)^2(Q_i)\Big)^{\frac 12}\lesssim h^{k}\|u\|_{k+3}.
\end{equation}
   Here $N_x,N_y,M_x,M_y$ denote the cardinalities of ${\mathcal E}_x^l,{\mathcal E}_y^l,{\mathcal E}_x^g,{\mathcal E}_y^g$, respectively.

\end{enumerate}
 \end {theorem}
 \begin{proof}  First, by choosing $\mu=k+1$ in \eqref{e-17}, we get 
   \begin{eqnarray*}
         \|\triangle\xi\|_0+\|\xi_{xyy}\|_0+\|\xi_{xxy}\|_0\lesssim h^{k}\|u\|_{k+3}. 
    \end{eqnarray*}
        By using  \eqref{eqq:1} and the orthogonality $a(\xi+\eta,v)=0$ for all $v\in W_h$,    we have 
    \[
       \| \xi\|^2_0\lesssim h^{2k+4}\|u\|^2_{k+3}+|a(\eta,{\mathcal I}_h\varphi)|=h^{2k+4}\|u\|^2_{k+3}+|a(E^xu+E^yu-E^xE^yu,{\mathcal I}_h\varphi)|. 
    \]
   Here $\varphi$ is the solution of \eqref{dual:problem} with $v= \xi$, and ${\mathcal I}_h\varphi\in\mathbb Q_1$ denotes the bilinear interpolation function of $\varphi$. 
 Noticing  that $E^xu\bot \mathbb P_{0}(x), \partial_xE^xu\bot\mathbb P_{1}(x)$ for $k\ge 4$,  then 
 \begin{eqnarray*}
    | a(E^xu,{\mathcal I}_h\varphi)|=|(-\alpha E^xu_{yy}+\beta_2E^xu_y+\gamma E^xu, {\mathcal I}_h\varphi-\bar \varphi)|\lesssim h^{k+2}\|u\|_{k+3}\|\varphi\|_1,
 \end{eqnarray*}
   where $\bar\varphi$ denotes the cell average of $\varphi$.  As for $k=3$,  we  use the integration by parts to obtain 
  \begin{eqnarray*}
    | a(E^xu,{\mathcal I}_h\varphi)|&=&|(-\alpha E^xu_{yy}+\beta_2E^xu_y+\gamma E^xu, {\mathcal I}_h\varphi)-(\beta_1E^xu,\partial_x{\mathcal I}_h\varphi)|\\
    &\lesssim& h^{k+1}\|u\|_{k+3}\|\varphi\|_1.
 \end{eqnarray*}
  Consequently,  
 \[
     |a(E^xu,{\mathcal I}_h\varphi)|\lesssim h^{\min {(k+2,2k-2)}} \|u\|_{k+3}\|\varphi\|_1.
 \] 
   Similarly, there holds 
 \begin{eqnarray*}
    | a(E^yu,{\mathcal I}_h\varphi)| +  | a(E^xE^yu,{\mathcal I}_h\varphi)|  \lesssim h^{\min {(k+2,2k-2)}} \|u\|_{k+3}\|\varphi\|_1,
 \end{eqnarray*}  
   and thus 
\begin{eqnarray}\label{eq:9}
     | a(\eta,{\mathcal I}_h\varphi)|\lesssim h^{\min {(k+2,2k-2)}} \|u\|_{k+3}\|\varphi\|_1,
 \end{eqnarray}
    which yields, together with the $H^2$ regularity $\|\varphi\|_2\lesssim \|\xi\|_0$, 
\[
    \| \xi\|_0\lesssim   h^{\min {(k+2,2k-2)}} \|u\|_{k+3}. 
\]

    We next estimate $\|\nabla \xi\|_0$. Given any $\zeta\in [C^1(\Omega)]^2$, let $\psi$ be the solution of the following
   dual problem
 \begin{eqnarray*}
   -\nabla\cdot(\alpha \nabla \psi)-\beta\cdot\nabla \psi+\gamma \psi=-\nabla\cdot\zeta\ \ \rm{in}\   \ \Omega,\
     {\rm and}\  \psi=0,  \ \rm{on}\  \  \partial\Omega.
\end{eqnarray*}
     By using the integration by parts,
 \begin{eqnarray*}
    (\nabla  \xi, \zeta)=-( \xi, \nabla\cdot\zeta)&=&(\xi,-(\nabla\cdot(\alpha \nabla \psi)-\beta\cdot\nabla \psi+\gamma \psi)\\
              &=&(-\nabla\cdot(\alpha \nabla  \xi )+\beta \cdot\nabla \xi+\gamma   \xi,\psi-{\mathcal I}_h\psi+{\mathcal I}_h\psi)\\
              &\lesssim & h \|\xi\|_{2}\|\psi\|_1+|a( \xi,{\mathcal I}_h\psi)|=h \|\xi\|_{2}\|\psi\|_1+|a( \eta,{\mathcal I}_h\psi)|.
\end{eqnarray*}
  Consequently, by the regularity result $\|\psi\|_1\lesssim \|\nabla\cdot \zeta\|_{-1}\lesssim \|\zeta\|_0$,  \eqref{eq:9}, and 
  the estimate of  $\|\xi\|_2$,  we derive 
\[
       |(\nabla  \xi, \zeta)|\lesssim  h^{k+1} \|u\|_{k+3}\|\|\zeta\|_0.
\]
  Since the set of all such $\zeta$ is dense in $L^2(\Omega)$,  the above inequality indicates that
\begin{equation}\label{eq:18}
     \|\nabla \xi\|_0\lesssim h^{k+1}\|u\|_{k+3}.
\end{equation}
   This finishes the proof of \eqref{e-18}.   
  
  In light of the superconvergence properties of $u_I$ in Proposition \ref{proposition:1},   we have 
   \begin{eqnarray*}
    e_{u,J}
    &\lesssim &\left(\frac {1}{MN}\sum_{i=1}^M\sum_{j=1}^N\| \xi\|^2_{0,\infty, \tau_{i,j}}\right)^{\frac 12}+h^{k+2}\|u\|_{k+3}.\\
    &\lesssim&  \| \xi\|_0+h^{k+2}\|u\|_{k+3}\lesssim  h^{\min {(k+2,2k-2)}} \|u\|_{k+3}. 
   \end{eqnarray*}
  Then \eqref{super:pg1} follows. 
  Similarly,   there hold
\[
   e_{\nabla u,l}\lesssim \|\nabla \xi\|_{0}+h^{k+1}\|u\|_{k+3},\  \
  e_{\triangle u,g}\lesssim \|\triangle \xi\|_{0}+h^{k}\|u\|_{k+3}.
\]
 Then \eqref{super:k3}-\eqref{super:k4} follow from the estimates of $\| \nabla \xi\|_0$ and $\| \xi\|_{2}$ directly.
 This finishes our proof. $\Box$
\end{proof}

  In the following, we study the highest superconvergence result of the $C^1$ Petrov-Galerkin approximation at the mesh nodes.  We use the idea of correction function to
   achieve our superconvergence goal. 
   The basic idea   of the correction function  is the construction of a specially designed  function $w_h\in V_h^0$
   such that $\tilde u_I=u_I-w_h$ is superconvergent towards the numerical solution $u_h$ in some norms, e.g., $H^2$ or $L^2$-norm, with higher order of accuracy. 
   
    Denote 
 \[
    \tilde \xi=\tilde u_I-u_h=u_I-w_h-u_h. 
 \]
In light of \eqref{eqq:1}-\eqref{eqq:2}, the errors $\|\tilde \xi\|_0$ and $\|\tilde\xi\|_2$ are dependent on two terms: $a(\tilde\xi,\tilde\xi_{xxyy})$ and
  $a(\tilde\xi,{\mathcal I}_h\varphi)$. By the orthogonality, we have
\begin{equation}\label{orth:1}
   a(\tilde\xi,\theta)=-a(\eta+w_h,\theta),  \ \ \forall \theta\in W_h. 
\end{equation}
  In other words, to achieve our superconvergence goal, the function $w_h\in V_h$ should be specially construct such that
\begin{equation}\label{decom:2}
   a(\eta,\theta)+a(w_h,\theta)=a(E^xu+E^yu-E^xE^yu,\theta)+a(w_h,\theta),\ \ \ \forall \theta\in W_h
\end{equation}
is of high order. Note that if we choose $w_h=0$, then we  get the superconvergence results presented in Theorem \ref{theo:02}, which is one order higher than the counterpart optimal convergence rate.

The next Proposition shows the existence of the correction function $w_h$, which satisfies our superconvergence goal. 

\begin{proposition}\label{proposition:2}
 Let $u\in W^{2k+1,\infty}(\Omega)$. There  exists a $w_h\in V^0_h$ such that
\begin{align}\label{esti:ww}
   \|w_h\|_{0,\infty}\lesssim h^{\min{(k+2,2k-2)}}\|u\|_{2k+1,\infty},\
   \|w_h\|_{1,\infty}+h\|w_h\|_{2,\infty}\lesssim h^{k+1}\|u\|_{2k+1,\infty}, \end{align}
   \begin{align}\label{esti:ww1}
    |w_h(x_i,y_j)|+|\nabla w_h(x_i,y_j)|\lesssim h^{2k-2}\|u\|_{2k+1,\infty}.
\end{align}
  Furthermore, there holds for any $\theta\in W_h$
\begin{eqnarray}\label{eqq:8}
    |a(u-u_I+w_h,\theta)|\lesssim h^{2k-2}\|u\|_{2k+1,\infty}\|\theta\|_{0}.
\end{eqnarray}
\end{proposition}

The proof of Proposition \ref{proposition:2} is given in the Appendix.

   Now we are ready to present the superconvergence of  $u_h$ at mesh nodes.

 \begin{theorem}\label{theo:2}
 Assume that $u\in W^{2k+1,\infty}(\Omega)$ is the solution of \eqref{con_laws}, and $u_h$ is the solution of \eqref{PG}.
Then 
\begin{equation}\label{averagenode}
   e_{u,n}\lesssim h^{2k-2}\|u\|_{2k+1,\infty},\ \
   e_{\nabla u,n}\lesssim  h^{2k-2}\|u\|_{2k+1,\infty},
\end{equation}
 where
\begin{eqnarray*}
   e_{v,n}=\Big(\frac {1}{{{MN}}}\sum_{i=1}^{{M-1}}\sum_{j=1}^{ {N-1}}\big(v-v_h\big)^2\big(x_{i},y_j\big)\Big)^{\frac 12},\ \ v=u,\nabla u. 
 \end{eqnarray*}
 \end {theorem}
 \begin{proof}  
      By \eqref{eqq:1}-\eqref{eqq:2}, \eqref{orth:1} and \eqref{eqq:8},  we have 
\begin{eqnarray*}
 \|\triangle \tilde \xi\|^2_0+\|\tilde \xi_{xxy}\|^2_0+\|\tilde \xi_{xyy}\|_0^2&\lesssim& |a(\eta+w_h,\tilde\xi_{xxyy})|+|a(\eta+w_h,{\mathcal I}_h\varphi)| \\
     &\lesssim & h^{2k-3}\|u\|_{2k+1,\infty}\|\tilde\xi_{xyy}\|_0 +h^{2k-2}\|u\|_{2k+1,\infty}\|{\mathcal I}_h\varphi\|_0, 
\end{eqnarray*}
where $\varphi$ is the solution of \eqref{dual:problem} with $v=\tilde \xi$, and in the last step, we have used the inverse inequality $\|\tilde \xi_{xxyy}\|_0\lesssim h^{-1}\|\tilde \xi_{xyy}\|_0$.
 Using the inequality $\|{\mathcal I}_h\varphi\|_0\le \|\varphi\|_2 \lesssim \|\tilde\xi\|_0\lesssim  \|\triangle\tilde\xi\|_0$, we immediately get
 \[
   \|\triangle \tilde \xi\|_0+\|\tilde \xi_{xxy}\|_0+\|\tilde \xi_{xyy}\|_0\lesssim  h^{2k-3}\|u\|_{2k+1,\infty}. 
 \]
 Then we follow the same argument as what we did in Theorem \ref{theo:02} to obtain 
\[
        \|\tilde \nabla \xi\|_0 \lesssim h^{2k-2}\|u\|_{2k+1,\infty},\ \ \ 
        \|\tilde  \xi\|_0\lesssim  \|\tilde \nabla \xi\|_1\lesssim h^{2k-2}\|u\|_{2k+1,\infty}. 
\]
    By the property of $u_I$ and \eqref{esti:ww1}, we get
 \begin{eqnarray*}
   \left|(u-u_h)(x_{i},y_j)\right|&=&\left|(u_I+w_h-u_h)(x_{i},y_j)-w_h(x_{i},y_j)\right|\\
   &\le &\|\tilde \xi\|_{0,\infty,\tau_{i,j}}+h^{2k-2}\|u\|_{2k+1,\infty},  
\end{eqnarray*}
  and thus, 
\begin{eqnarray*}
    e_{u,n}
    &\lesssim&  \|\tilde \xi\|_0+h^{2k-2}\|u\|_{2k+1,\infty}\lesssim h^{2k-2}\|u\|_{2k+1,\infty}.
\end{eqnarray*}
  Following the same argument, we have
\begin{equation}\label{eq:17}
   e_{\nabla u,n}\lesssim \|\nabla\tilde \xi\|_0+h^{2k-2}\|u\|_{2k+1,\infty}\lesssim  h^{2k-2}\|u\|_{2k+1,\infty}.
\end{equation}
Then \eqref{averagenode} follows.  This finishes our proof. $\Box$
\end{proof}

   With the help of the correction function $w_h$, we can also improve our superocnvergence results from the average sense to 
   the point-wise sense for $k\ge 4$. 
  
\begin{theorem}\label{theo:03}
    Suppose all the conditions of Theorem \ref{theo:2} hold true. Then
  \begin{equation}\label{super:pg1}
  |(u-u_h)(x_i,y_j)|\lesssim h^{2k-2}|{\rm ln} h|^{\frac 12}\|u\|_{2k+1,\infty},\ \ 
  |(u-u_h) (P) | \lesssim  h^{k+2}\max(1,h^{k-4}{\rm ln}h^{\frac 12})\|u\|_{2k+1,\infty},
\end{equation}
\begin{eqnarray}\label{super:pg2}
 |\partial_x(u-u_h) (P_1)| + |\partial_y(u-u_h) (Q_1)| \lesssim  h^{k+1}\|u\|_{2k+1,\infty},
\end{eqnarray} 
 \begin{equation}\label{approx:super44}
   |\partial_{xx}^2(u-u_h)(P_2)|+|\partial_{yy}^2(u-u_h)(Q_2)| +|\partial_{xy}^2(u-u_h)(P_3)| \lesssim h^k\|u\|_{2k+1,\infty},
\end{equation}
  where $P\in{\mathcal R},P_1\in {\mathcal E}_x^l, Q_1\in {\mathcal E}_y^l$,  $P_2\in  {\mathcal E}_x^g, Q_2\in {\mathcal E}_y^g $ and 
$P_3\in{\mathcal L}$. 
\end{theorem}
 \begin{proof} 
   We first define the $C^0$-conforming finite element space $S_h$ as follows:
\[
    S_h=\{ v\in C^0(\Omega): \; v|_{\partial\Omega}=0, v|_{\tau}\in \mathbb{Q}_k(x,y)=\mathbb {P}_k(x)\times \mathbb {P}_k(y), \tau\in{\mathcal T}_h\}^{}.
\]
   We denote by $a_e(\cdot,\cdot)$ the bilinear form of the finite element method, that is,
\[
   a_e(u,v)=(\alpha \nabla u,\nabla v)+({\bf \beta}\nabla u,v)+(\gamma u,v).
\]
  Note that $a_e(u,v)$ is coercive and continuous in the $H_0^1$ space. By Lax-Milgram Lemma, there exists a $g_h\in S_h$ such that
\begin{equation}\label{eq:12}
   a_e(v_h,g_h)=v_h(x,y),\ \ \forall v_h\in S_h.
\end{equation}
 Especially, we choose $v_h=g_h$ to obtain
\begin{equation}
   \|g_h\|^2_{1}\lesssim |a_e(g_h,g_h)|=|g_h(x,y)|\le \|g_h\|_{0,\infty}.
\end{equation}
Since (cf.,[31], p.84, Theorem 2.8)
\[
   \|v_h\|_{0,\infty}\lesssim |{\rm ln} h|^{\frac 12}\|v_h\|_1,\ \ \forall v_h\in S_h,
\]
   we have
\[
    \|g_h\|_{1}\lesssim |{\rm ln} h|^{\frac 12}.
\]
   By choosing $v_h=\tilde \xi$ in \eqref{eq:12} and use the integration by parts and \eqref{eqq:8}, 
\begin{eqnarray*}
    \|\tilde \xi\|_{0,\infty}&\le&  |a_e(\tilde\xi,g_h)|
   = |a(\tilde \xi, g_h-R_hg_h)-a(\eta+w_h, R_hg_h)|\\
    &\lesssim &h \|\tilde \xi\|_2\|g_h\|_1+h^{2k-2}\|g_h\|_0\|u\|_{2k+1,\infty}
    \lesssim h^{2k-2}|{\rm ln} h|^{\frac 12}\|u\|_{2k+1,\infty}.
\end{eqnarray*}
  Here $R_h$ denotes the $L^2$ projection of $S_h$ onto $W_h$. 
 Then the desired results \eqref{super:pg1}-\eqref{approx:super44} follow from the approximation properties of $u_I$ and the estimates of $w_h$ in Proposition \ref{proposition:2}. The proof is complete.  $\Box$
\end{proof}

\section{Interior estimates for the $C^1$ Petrov-Galerkin method}

   In this section, we study interior a priori error estimates in $H^2, H^1, L^2$-norms, which can be estimated with   an  error  in a strong norm  on a smaller domain plus 
an  error in a weaker norm over a slightly larger domain.   We begin with some preliminaries. 

Let $\Omega_0\subset\subset\Omega_1\subset\subset\Omega_2\subset\subset\cdots\subset\subset\Omega_m\subset\subset\Omega$ be separated by $d\ge c_0h$, with $\Omega_i, i\le m$ the 
rectangular domain.  For any domain $D$, we define 
\begin{eqnarray*}
    &&W^0_h(D):=\{ v\in W_h:  v|_{\partial D}=0,\  {\rm supp} \ v \subset\bar D\}^{},\\
    && V^0_h(D):=\{ v\in V_h:  v|_{\partial D}=0,\  {\rm supp} \ v \subset\bar D\}^{}.   
\end{eqnarray*}
   Define 
 \begin{equation}\label{e-11}
     \||v\||^2_{D} : = \|\triangle v\|^2_{0,D}+ \|\partial_x\partial_{yy} v\|^2_{0,D}+ \|\partial_y\partial_{xx} v\|^2_{0,D}. 
 \end{equation}
 
 Denote by $B(u,v)$  the bilinear form which is defined as 
\[
    B(u,v):=a(u,v_{xxyy}). 
\] 

\begin{lemma} Let  $\Omega_0\subset\subset\Omega'$ and $p\ge 0$ be a fixed but arbitrary integer.  
  Suppose $\bar e \in V^0_h(\Omega')$ is the solution of the problems 
 \[
    a(\bar e, \zeta)=0,\ \ \forall \zeta\in W_h(\Omega')\ {\rm or}\  B(\bar e, \theta_{xxyy})=0,\ \ \forall \theta\in V_h^0(\Omega'). 
 \]
 Then 
\begin{eqnarray}\label{e-14}
    &&\|\bar e\|_{1,\Omega_0 }\lesssim h \|\triangle \bar e\|_{0,\Omega'}+\|  \bar e \|_{-p,\Omega'},\ \ \|\bar e\|_{0,\Omega_0 }\lesssim h^2 \|\triangle \bar e\|_{0,\Omega'}+\|  \bar e \|_{-p,\Omega'},\\\label{e-15}
    &&\|\bar e \|_{2,\Omega_0} \lesssim h^{\frac 12} \|| \bar e \||_{0,\Omega'}+\|  \bar e \|_{-p,\Omega'}. 
\end{eqnarray}
 \end{lemma}
\begin{proof} We only consider the case that  $\bar  e$ satisfies $a(\bar e, \zeta)=0$ since the same argument can be applied to the case in which $B(\bar e, \theta_{xxyy})=0$. 

 For any $s\le 1, v\in H^{-s}(\Omega_1)$, denote by $\varphi\in H^{-s+2}(\Omega_1)$ the solution of \eqref{dual:problem}.  Let $w=1$ on $\Omega_0$ and $w\in C_0^{\infty}(\Omega_1)$ with $\Omega_1\subset\subset\Omega'$. 
 Then 
 \begin{eqnarray*}
     \|w\bar e\|_{s,\Omega_1}&\lesssim& \sup_{v\in H^{-s}_0(\Omega_1)}\frac{|(w\bar e, v)|}{\|v\|_{-s,\Omega_1}}=\sup_{\varphi\in H^{2-s}_0(\Omega_1)}\frac{|a(w\bar e, \varphi)|}{\|\varphi\|_{2-s,\Omega_1}}\\
     &=&\sup_{\varphi\in H^{2-s}_0(\Omega_1)}\frac{|a(\bar e, w\varphi)+I|}{\|\varphi\|_{2-s,\Omega_1}}
   = \sup_{\varphi\in H^{2-s}_0(\Omega_1)}\frac{|a(\bar e,  w\varphi-{\mathcal I}_h(w\varphi))+I|}{\|\varphi\|_{2-s,\Omega_1}}, 
 \end{eqnarray*}
  where ${\mathcal I}_h(w\varphi)\in{\mathbb Q}_1$ denotes the bilinear function of $w\varphi$,  and 
\[
  | I |=|(-\alpha \triangle w\bar e-2\alpha\nabla w\nabla \bar e,\varphi)|=|(2\alpha\nabla\cdot(\nabla w\varphi)-\alpha\triangle w, \bar e)|\lesssim \|\bar e\|_{s-1,\Omega_1}\|\varphi\|_{2-s,\Omega_1}. 
\]
  Consequently, 
\begin{eqnarray*}
 \|\bar e\|_{s,\Omega_0}&\le &   \|w\bar e\|_{s,\Omega_1}   \lesssim   \sup_{\varphi\in H^{2-s}_0(\Omega_1)}
    \frac{ h^{\min(2-s,2)}( \|\triangle \bar e\|_{0,\Omega_1}+ \| \bar e\|_{1,\Omega_1})\|\varphi\|_{2-s,\Omega_1}}{\|\varphi\|_{2-s,\Omega_1}}+\|\bar e\|_{s-1,\Omega_1}\\
     & \lesssim &  h^{\min(2-s,2)}( \|\triangle \bar e\|_{0,\Omega_1}+ \| \bar e\|_{1,\Omega_1})+\|\bar e\|_{s-1,\Omega_1}. 
\end{eqnarray*}
 Especially, by choosing $s=0$ and iterating the above inequality $p$ times, we get 
\begin{eqnarray*}
   \|\bar e\|_{0,\Omega_0}&\lesssim & h^2( \|\triangle \bar e\|_{0,\Omega_1}+ \| \bar e\|_{1,\Omega_1})+\|\bar e\|_{-1,\Omega_1}\\
   &\lesssim & h^2( \|\triangle \bar e\|_{0,\Omega_2}+ \| \bar e\|_{1,\Omega_2})+\|\bar e\|_{-2,\Omega_2}\\
   &\lesssim & h^2( \|\triangle \bar e\|_{0,\Omega_p}+ \| \bar e\|_{1,\Omega_p})+\|\bar e\|_{-p,\Omega_p}. 
\end{eqnarray*}
  Similarly, we choose $s=1$ to obtain 
 \begin{eqnarray*}
     \|\bar e\|_{1,\Omega_0}\lesssim 
     h (\|\triangle \bar e\|_{0,\Omega_{p+1}}+\| \bar e\|_{1,\Omega_{p+1}})
      +\|\bar e\|_{-p,\Omega_{p+1}}. 
 \end{eqnarray*}
 Let $\Omega_{p+1}\subset\subset\Omega_{p+2}\subset\subset\cdots\subset\subset\Omega_{2p}=\Omega'$ and iterate the above inequality $p$ times, we obtain 
\begin{eqnarray}\label{e-16}
  \|\bar e\|_{1,\Omega_0}\lesssim h \|\triangle \bar e\|_{0,\Omega_2p}
 +h^{p+1}\|\bar e\|_{1,\Omega_{2p}}+\|  \bar e \|_{-p,\Omega_{2p}}.
\end{eqnarray} 
      Then \eqref{e-14} follows by using the inverse inequality.

    We next estimate $\|\triangle e\|_{0,\Omega_0}$.  
  Recalling the definition of the bilinear form of $a(\cdot,\cdot)$, we have 
\begin{eqnarray*}
  \|\triangle (w\bar e)\|^2_{0,\Omega_1} &\lesssim & | a(w\bar e, \triangle (w\bar e))| + \|w\bar e\|^2_{1,\Omega_1}\\
       &=&|a(e, w\triangle (w\bar e))+(-\alpha \bar e \triangle w -2\alpha\nabla w\nabla \bar e, \triangle (w\bar e))| + \|w\bar e\|^2_{1,\Omega_1}\\
       &\lesssim &| a(\bar e, w\triangle (w\bar e)-\theta)|+\|w\bar e\|^2_{1,\Omega_1}\\
       &\lesssim &( \|\triangle \bar e\|_{0,\Omega_1}+ \|\bar e\|_{1,\Omega_1})\|w\triangle (w\bar e)-\theta\|_{0,\Omega_1}+\|\bar e\|^2_{1,\Omega_1},\ \ \forall \theta\in W_h(\Omega_1). 
 \end{eqnarray*}  
     By the standard approximation theory, there holds 
\begin{eqnarray*}
    \| w\triangle (w\bar e)-\theta\|_{0,\Omega_1} & \lesssim & h^{k-1}( \|\triangle \bar e\|_{k-1}+\|\nabla \bar e\|_{k-1}+\|\bar e\|_{k-1}) \\
    &\lesssim & h^{k-1} (\|\partial_y^{k-1}\partial_{xx}\bar e\|_{0,\Omega_1}+ \|\partial_x^{k-1}\partial_{yy}\bar e\|_{0,\Omega_1}+\|\triangle \bar e\|_{k-2,\Omega_1})\\
    &\lesssim & h(\||\bar e\||_{\Omega_1}+\|\bar e\|_{1,\Omega_1}). 
\end{eqnarray*}  
   Here in the last step,  we have used the inverse inequality. 
 Consequently, 
\[
    \| \bar e \|^2_{2,\Omega_0} \le  \| w\bar e \|^2_{2,\Omega_1} \lesssim \|\triangle (w\bar e)\|^2_{0,\Omega_1} \lesssim h \||\bar e\||^2_{\Omega_1}+ \|\bar e\|^2_{1,\Omega_1}. 
\]
 which yields ( together with \eqref{e-14}) the desired result \eqref{e-15}. The proof is complete.  $\Box$
\end{proof}

   Given any $v\in C^1(\Omega_1)\cap H^4_0(\Omega_1)$,   let  $Pv\in V_h^0(\Omega_1)$ and $P^*v\in V_h^0(\Omega_1)$ be defined as the solutions of the equations
 \begin{eqnarray}\label{e:2}
     B(v-Pv,\varphi)=0,\ \ B(\varphi, v-P^*v)=0,\ \ \forall \varphi\in V_h^0(\Omega_1). 
 \end{eqnarray}
   By the same argument as what we did in Theorem \ref{theo:4}, we can prove that $Pv$ and $P^*v$ are uniquely defined. 
   
    In light of the conclusions in Lemma \ref{lemma-1}, we easily obtain,  by   using \eqref{e:2} and the Cauchy-Schwarz inequality,  
     the integration by parts and the homogenous boundary condition $v|_{\partial\Omega_1}=0$, 
  \begin{eqnarray}\label{e:1}
 \begin{split}
    \||Pv\||_{\Omega_1}^2&\lesssim 
   B(Pv,Pv)+a(Pv,{\mathcal I}_h\varphi)  =B(v,Pv)+a(v,{\mathcal I}_h\varphi) &\\
   &\lesssim \|| v\||_{\Omega_1}\||Pv\||_{\Omega_1}+\|v\|_{2,\Omega_1}\|\varphi\|_{0,\Omega_1}, &
 \end{split}
 \end{eqnarray}  
  where  $\varphi$ is the solution of \eqref {dual:problem} with $v$ replaced by $Pv$, and in the second step, we  have used the identity 
 \begin{equation}\label{e-13}
     a(v,{\mathcal I}_h\varphi)=B(v, \varphi_1)=B(Pv, \varphi_1)=a(Pv,{\mathcal I}_h\varphi)
 \end{equation}
   with $  \varphi_1\in V_h^0$ satisfying  $\partial_{xx}\partial_{yy}\varphi_1={\mathcal I}_h\varphi. $
  Using the $H^2$ regularity assumption $\|\varphi\|_{0,\Omega_1}\lesssim \|Pv\|_{0,\Omega_1}\lesssim \|\triangle Pv\|_{0,\Omega_1}$, we get 
\[
   \||Pv\||_{\Omega_1}^2\lesssim \||v\||^2_{\Omega_1}+\|v\|^2_2\lesssim   \||v\||^2_{\Omega_1}. 
\]
  Similarly, we can prove that the same result holds true for $P^*v$. 
 
 Let $w=1$ on $\Omega_0$ and $w\in C_0^{\infty}(\Omega')$  with $\Omega_0\subset\subset \Omega'$.  Set $\tilde u=wu$ and  denote $\tilde e=\tilde u-P\tilde u$.   By using   \eqref{eqq:1}-\eqref{eqq:2},  \eqref{e:2} and the integration by parts, we get 
\begin{eqnarray*}
     \||\tilde e \||_{\Omega'}^2 & \lesssim & |a(\tilde e, \tilde e_{xxyy})+a(\tilde e, {\mathcal I}_h\varphi)|=|B(\tilde e, \tilde e)|\\
     &= & |B(\tilde e, \tilde e-P^*\tilde e)|=|B(\tilde u-\tilde u_I, \tilde e-P^*\tilde e)| \\
     &\lesssim & \||\tilde u-\tilde u_I\||_{\Omega'}\||\tilde e-P^*\tilde e\||_{\Omega'} \lesssim h^{\mu-1}\|\tilde u\|_{\mu+2,\Omega'},\ \mu\le k. 
\end{eqnarray*}
 Consequently, 
 \begin{equation}\label{e-6}
   \|| u-P\tilde u\||_{\Omega_0}\lesssim  \||\tilde e \||_{\Omega'} \lesssim h^{\mu-1}\|\tilde u\|_{\mu+2,\Omega'}. 
 \end{equation}
  We next estimate $\||P\tilde u-u_h\||_{\Omega_0}$.  
  
\begin{lemma}\label{lemma-0} 
 Assume that $\Omega_0\subset\subset\Omega'$ and $p\ge 0$ is a fixed but arbitrary integer.  Let $u_h$ be the solution of \eqref{PG}, 
  $\tilde u=wu$ with $w=1$ on $\Omega_0$ and $w\in C_0^{\infty}(\Omega')$, and $P\tilde u$ be  defined by \eqref{e:2}. 
  Then for sufficiently small $h$, 
 \begin{equation}\label{e-7}
   \||P\tilde u-u_h\||_{\Omega_0}\lesssim  \|P\tilde u-u_h\|_{-p,\Omega'}. 
 \end{equation}
 \end{lemma}
\begin{proof} First, we note that 
\[
   B(u-u_h,v)=0,\ \ \forall v\in V_h^0(\Omega_0), 
\]
    which yields (together with  \eqref{e:2}) 
 \begin{equation}\label{e:3}
   B(u_h-P\tilde u,v)=a( u_h-P\tilde u,v_{xxyy})=0,\ \ \forall v\in V_h^0(\Omega_0). 
\end{equation}
  Let $\bar e=u_h- P\tilde u$.  Then 
 \begin{equation}\label{e-2}
    \||\bar e\||_{\Omega_0}\le \||w\bar e\||_{\Omega_1} \le \|| (w \bar e)-P(w \bar e) \||_{\Omega_1}+\|| P(w \bar e) \||_{\Omega_1}. 
 \end{equation}
   As for $\|| (w \bar e)-P(w \bar e) \||_{\Omega_1}$, we have from \eqref{eqq:1}-\eqref{eqq:2}, \eqref{e:2}  and the integration by parts that 
  \begin{eqnarray*}
   \|| (w \bar e)-P(w \bar e) \||^2_{\Omega_1}&\lesssim & B((w \bar e)-P(w \bar e), (w \bar e)-P(w \bar e))\\
     &= & B((w \bar e)-P(w \bar e), (w \bar e)-(w \bar e)_I)\\
     &\lesssim &   \|| (w \bar e)-P(w \bar e) \||_{\Omega_1}\|| (w \bar e)-(w \bar e)_I \||_{\Omega_1}. 
  \end{eqnarray*}
     Here $(w \bar e)_I$ denotes the  truncated Jacobi projection of $ w\bar e$. 
   From the property of $u_I$ in  Lemma \ref{lemma:1}, we derive  
 \begin{eqnarray*}
    \||u-u_I \||_{\Omega_1} \lesssim   \|(u-u_I )_{xxy}\|_{0,\Omega_1}+\|(u-u_I )_{xyy}\|_{0,\Omega_1} \lesssim h^{k-1}\|u\|_{k+2,\Omega_1}.
 \end{eqnarray*}
   Consequently, 
 \begin{eqnarray}\label{e-3}
 \begin{split}
      \|| (w \bar e)-P(w \bar e) \||_{\Omega_1}   &\lesssim    \|| (w \bar e)-(w \bar e)_I \||_{\Omega_1}\lesssim h^{k-1}\|w \bar e\|_{k+2,\Omega_1} & \\
        &\lesssim   h^{k-1}\|\bar e\|_{k,\Omega_1}\lesssim h ||\bar e\|_{2, \Omega_1}.  &
  \end{split}
 \end{eqnarray}
    Here in the last step, we have used the inverse inequality $\|\bar e\|_{k,\Omega_1}\lesssim h^{2-k}\|\bar e\|_{2,\Omega_1}.$
  
   Let $\varphi$ be the solution of \eqref{dual:problem} with $v=P(w\bar e)$.   Following the same argument as what we did in \eqref{e:1}, we derive 
 \begin{eqnarray*}
  \||P(w\bar e)\||_{\Omega_1}^2 &\lesssim & \left| B(w\bar e,P(w\bar e))+a(w\bar e,{\mathcal I}_h\varphi)\right|\\
  &=& \left|  B(\bar e,wP(w\bar e))+a(\bar e, w{\mathcal I}_h\varphi)+I \right|\\
  &=& \left| B(\bar e,wP(w\bar e) -(wP(w\bar e))_I)+a(\bar e, w{\mathcal I}_h\varphi-{\mathcal I}_h(w{\mathcal I}_h\varphi))+I\right|\\
   &\lesssim &   h \||\bar e\||_{0,\Omega_1}(\|wP(w\bar e)\|_{2,\Omega_1}+\|w{\mathcal I}_h\varphi\|_{1,\Omega_1})+|I|, 
  \end{eqnarray*}
  where in the third step and last step,  we have used \eqref{e:3} and the integration by parts, respectively, and 
 \begin{eqnarray*}
    I =\int_{\Omega_1} \left(-\alpha \triangle w\bar e-2\alpha\nabla w\nabla \bar e\right)\big( (P(w \bar e))_{xxyy}+{\mathcal I}_h\varphi\big)dxdy.
 \end{eqnarray*}
   Again we use the integration by parts and the Cauchy-Schwarz inequality to obtain 
  \[
     |I|\lesssim \|  w\bar e\|_2( \||P(w \bar e)\||_{\Omega_1}+\|\varphi\|_{0,\Omega_1}), 
  \]
    which yields, together with the $H^2$ regularity  $\|\varphi\|_{2,\Omega_1}\lesssim  \|P(w \bar e)\|_{0,\Omega_1}\lesssim  \||P(w \bar e)\||_{\Omega_1}$, 
\begin{eqnarray}\label{e-4}
   \||P(w\bar e)\||_{\Omega_1} \lesssim h\||\bar e\||_{\Omega_1}+\|\bar e \|_{2,\Omega_1}. 
\end{eqnarray}
   Substituting \eqref{e-3}-\eqref{e-4}   into \eqref{e-2}  and using \eqref{e-15} with $\Omega_0,\Omega'$ replaced by $\Omega_1,\Omega_2$, we get
\[
   \||\bar e\||_{\Omega_0}\lesssim h^{\frac 12}\||\bar e\||_{\Omega_2}+\|\bar e \|_{-p,\Omega_2}\lesssim h\||\bar e\||_{\Omega_3}+\|\bar e \|_{-p,\Omega_3}. 
\]
  By integrating  the above inequality $p+2$ times and using the inverse inequality again, we obtain 
\begin{eqnarray*}
   \||\bar e\||_{\Omega_0}
   \lesssim  h^{p+3}\||\bar e\||_{0,\Omega_{p+5}}+\|  \bar e \|_{-p,\Omega_{p+5}}\lesssim \|  \bar e \|_{-p,\Omega'}. 
\end{eqnarray*} 
      This finishes our proof.  $\Box$
  \end{proof}

   Now we are ready to present our interior estimates in all $H^2, H^1, L^2$-norms
    
\begin{theorem}
   Let $\Omega_0\subset\subset\Omega_1\subset\subset\Omega$, $u\in H^{l}(\Omega_1)$ and $u_h$ be the solutions of \eqref{con_laws} and  \eqref{PG}, respectively.  
   Suppose that $p\ge 0$ is a fixed but arbitrary interger. Then for $\mu\le \min(k,l-2)$, 
\begin{eqnarray}\label{e-9}
   &&  \||u-u_h\||_{\Omega_0} \lesssim h^{\mu-1} \|u\|_{\mu+2,\Omega_1}+  \| u-u_h\|_{-p,\Omega_1},\\   \label{ee-1}
   && \|u-u_h\|_{1,\Omega_0} \lesssim  h^{\mu} (\|u\|_{\mu+2,\Omega_1}+\|u\|_{1,\Omega})+  \| u-u_h\|_{-p,\Omega_1},\\\label{ee-2}
   && \|u-u_h\|_{0,\Omega_0} \lesssim  h^{\mu+1} (\|u\|_{\mu+2,\Omega_1}+\|u\|_{1,\Omega})+  \| u-u_h\|_{-p,\Omega_1}.
\end{eqnarray}     
   Furthermore, if $u\in H^3(\Omega)\cap H^{k+2}(\Omega_1)$, there hold the following optimal   interior estimates: 
\begin{equation}\label{e-8}
    \|u-u_h\|_{0,\Omega_0}+ h\|u-u_h\|_{1,\Omega_0} +h^2\||u-u_h\||_{\Omega_0} \lesssim h^{k+1} ( \|u\|_{k+2,\Omega_1}+  \| u\|_{3,\Omega}). 
\end{equation}
\end{theorem}
\begin{proof}
 Let $\Omega'\subset\subset\Omega_1$.  As a direct consequence of \eqref{e-6}-\eqref{e-7},  
\begin{eqnarray*}
    \||u-u_h\||_{\Omega_0}
     &\lesssim & h^{\mu-1} \|u\|_{\mu+2,\Omega_1}+ \|P\tilde u-u_h\|_{-p,\Omega'}\\
     &\lesssim & h^{\mu-1} \|u\|_{\mu+2,\Omega_1}+ \|P\tilde u-u\|_{-p,\Omega'}+ \| u-u_h\|_{-p,\Omega_1}\\
      &\lesssim & h^{\mu-1} \|u\|_{\mu+2,\Omega_1}+ \|P\tilde u-\tilde u\|_{1,\Omega'}+ \| u-u_h\|_{-p,\Omega_1}.
\end{eqnarray*}
  Replacing $\Omega_0,\Omega'$ by $\Omega', \Omega_1$ in \eqref{e-6}    yields 
\[
   \|P\tilde u-\tilde u\|_{1,\Omega'}\lesssim \||P\tilde u-\tilde u\||_{\Omega'}\lesssim h^{\mu-1} \|u\|_{\mu+2,\Omega_1}. 
\] 
  Then the desired result  \eqref{e-9} follows. 
  
  Note that $a(e,\zeta)=0$ for all $\zeta\in W_h$. Following the same argument as what we did in \eqref{e-16}, we have 
 \begin{eqnarray*}
   && \|e\|_{1,\Omega_0}\lesssim h \|\triangle  e\|_{0,\Omega'}
 +h^{p+1}\| e\|_{1,\Omega'}+\|  e \|_{-p,\Omega'},\\
 && \|e\|_{0,\Omega_0}\lesssim h^2 \|\triangle  e\|_{0,\Omega'}
 +h^{p+1}\| e\|_{1,\Omega'}+\|  e \|_{-p,\Omega'}.
 \end{eqnarray*}
   Then \eqref{ee-1} and \eqref{ee-2} follows from \eqref{optimal:1}.

   As for the term $\|u-u_h\|_{-p,\Omega_1}$, we  first suppose $\varphi$ is the solution of the problem \eqref{dual:problem}
  and $ \|\varphi\|_{p+2,\Omega}\lesssim \|v\|_{p,\Omega}. $
    Then from the integration by parts, 
 \begin{eqnarray*}
    \|e\|_{-p,\Omega_1} \le \|e\|_{-p,\Omega}&=&\sup\limits_{v\in C_{0}^{\infty} (\Omega)}\frac{|(e,v)|}{\|v\|_{p,\Omega}}=\sup\limits_{\varphi\in C_{0}^{\infty} (\Omega)}\frac{|a(e,\varphi)|}{\|\varphi\|_{p+2,\Omega}}\\ 
     &=&\sup\limits_{\varphi\in C_{0}^{\infty} (\Omega)}\frac{|a(e,\varphi-{\mathcal I}_{k-2}\varphi)|}{\|\varphi\|_{p+2,\Omega}}
     \lesssim h^{\min(k-1, p)}\|e\|_{2,\Omega}. 
 \end{eqnarray*}
   Substituting the estimate of $\|e\|_{-p,\Omega_1}$ and \eqref{optimal:1}  into \eqref{e-9}-\eqref{ee-2}, we obtain 
  \eqref{e-8} directly.  The proof is complete.  $\Box$
 \end{proof}

\begin{remark}
   The interior error estimates in \eqref{e-9}-\eqref{ee-2} indicates that  errors in the $H^2, H^1, L^2$-norms  over any compact subdomain $\Omega_0$ of $\Omega_1$ 
   may be estimated with  an almost optimal order of accuracy that is possible locally for the  subspace $V_h$ plus an error in a much weak norm $H^{-p}(\Omega)$. 
   Just as pointed out in \cite{c5}, the significance of the negative norm is that, under some very important circumstances, one can prove high order convergence rate in
    negative norms with relatively less  requirements on the  global smoothness of $u$. 
 \end{remark}

Following the same arguments, we can also obtain interior estimates for the error  $u_h-u_I$ in $H^2, L^2, H^1$-norms. 
For simplicity, we discuss only the interior error  $\||u_I-u_h\||_{\Omega_0}$.  Similar  argument   can be applied to estimating other norms by some tedious 
calculations. 

Note that 
 \begin{equation}
   B(u_I-u_h,v)=a( u_I-u,v_{xxyy}),\ \ \forall v\in V_h^0(\Omega_0). 
\end{equation}
  As we may observe, the only difference between the above equation and \eqref{e:3} lies in the right hand side. 
  Following the same argument as what we  did in Lemma \ref{lemma-0} and choosing $\mu=k+1$ in \eqref{e-10},  we get 
 \begin{eqnarray*}
      \||u_I-u_h\||_{\Omega_0} & \lesssim &  \|u_I-u_h\|_{-p,\Omega_1}+h^{k}\|u\|_{k+3,\Omega_1}\\
         & \lesssim &  \|u_I-u\|_{-p,\Omega_1}+h^{k}\|u\|_{k+3,\Omega_1}+\|e\|_{-p,\Omega_1}.  
 \end{eqnarray*}
   Using the error decomposition of $u-u_I$ and the properties of $u_I$ in Lemma \ref{lemma:1},  we have 
 \begin{eqnarray*}
    \||\eta\|_{-p,\Omega}=\sup\limits_{v\in C_{0}^{\infty} (\Omega)}\frac{|(\eta,v)|}{\|v\|_{p,\Omega}}=\sup\limits_{\varphi\in C_{0}^{\infty} (\Omega)}\frac{|a(\eta,\varphi)|}{\|\varphi\|_{p+2,\Omega}}
         \lesssim h^{\min(k-1, p)}\|\eta\|_{2,\Omega}. 
 \end{eqnarray*}
  Therefore, by choosing $p=k-1$ and using the error estimates  $\|\eta\|_2+\|e\|_2\lesssim h\|u\|_{3,\Omega}$, 
\[
   \||u_I-u_h\||_{\Omega_0} \lesssim h^k (\|u\|_{k+3,\Omega_1}+\|u\|_{3,\Omega}), 
\] 
 which indicates a superconvergence result for the interior error  $\||u_I-u_h\||_{\Omega_0}$.

\section{Numerical experiments}

In this section, we present some numerical examples to verify our theoretical  findings in previous sections.


In our experiments, we adopt the $C^1$ Petrov-Galerkin method \eqref{PG}
for the convection-diffusion equation \eqref{con_laws} with $k=3,4,5$, respectively.
We test various errors for $u-u_h$, including $e_{u,n}$  and $e_{\nabla u,n}$ defined in Theorem \ref{theo:2}, the maximum error on roots of $J^{-2,-2}_{k+1}(x)J^{-2,-2}_{k+1}(y)$, the derivative error on the Lobatto lines,
and the second order  derivative error on  the Gauss lines and Lobatto points,
 which are defined as:
\begin{equation*}
\begin{array}{rl}
&e_u = \max\limits_{P \in \mathcal R} |(u-u_h)(P)|,\\
&e_{\nabla u} = \max\limits_{P_1 \in {\mathcal E}_x^l} |\partial_x (u-u_h)(P_1)| + \max\limits_{Q_1 \in {\mathcal E}_y^l} |\partial_x (u-u_h)(Q_1)|, \\
&e_{\Delta u} = \max\limits_{P_2 \in {\mathcal E}_x^g} |\partial^2_{xx} (u-u_h)(P_2)| + \max\limits_{Q_2 \in {\mathcal E}_y^g} |\partial^2_{yy} (u-u_h)(Q_2)|
+ \max\limits_{P_3 \in \mathcal L} |\partial^2_{xy} (u-u_h)(P_3)|.
\end{array}
\end{equation*}
We obtain our meshes by dividing the domain into $M\times N$
 rectangles, which is generated by randomly and independently perturbing each node in the $x$ and $y$ axes of a uniform mesh as
$$x_i = \frac{i}{M} + \varepsilon \frac{1}{M} \sin(\frac{i\pi}{M})\text{randn}(),\quad 0\leq i\leq M,$$
$$y_j = \frac{j}{N} + \varepsilon \frac{1}{N} \sin(\frac{j\pi}{N})\text{randn}(),\quad 0\leq j\leq N,$$
where $\text{randn()}$ returns a uniformly distributed random number in $(0,1).$
If not otherwise stated, we choose $M=N$ and $\varepsilon = 0.001$.

{\it Example 1:}  We consider the problem \eqref{con_laws}
 and take the constant coefficients  as
\[
 \alpha=\gamma=1, \quad {\bf \beta}=(1,1).
\]
The right-hand side function $f(x,y)$ is chosen such that
 the exact solution  is
$$u(x,y) = \sin(\pi x)\sin(\pi y).$$

In Figure \ref{const1_coefficient}, we show error curves of various approximation  errors calculated from the $C^1$ Petrov-Galerkin method for $k=3, 4, 5$, respectively. 
We observe that both convergence rates for the function value error (i.e., $e_{u,n}$) and the first-order derivative error (i.e., $e_{\nabla u, n}$)  at mesh nodes can reach as high as $h^{2k-2}$. As for the errors $e_{u}$ (i.e., the function value error  at roots of the Jacobi polynomial $J^{-2,-2}_{k+1}(x)J^{-2,-2}_{k+1}(y)$),
 $e_{\nabla u}$ (i.e., the  first-order derivative error on  the Lobatto lines), $e_{\Delta u}$ (i.e., the second-order derivative error on    the  Gauss lines and Lobatto points),  convergence  rates are $h^{k+2}$, $h^{k+1}$, $h^k$, respectively. They are all
consistent with error bounds established  in Theorems \ref{theo:2}-\ref{theo:03}.
We also test the supercloseness between the $C^1$ Petrov-Galerkin solution $u_h$ and the Jacobi projection $u_I$.
As expected, the convergence  rates for errors $\|u_h-u_I\|_0, \|u_h-u_I\|_1, \|u_h-u_I\|_2$    are $h^{\min\{k+2,2k-2\}},$ $h^{k+1}$, $h^k$, respectively.
These results verify our theoretical findings \eqref{e-18} in Theorem \ref{theo:02}.

\begin{figure}[H]
\centering
\subfigure{
\begin{minipage}[t]{0.33\linewidth}
\centering
\includegraphics[width=1.6in,height=1.4in]{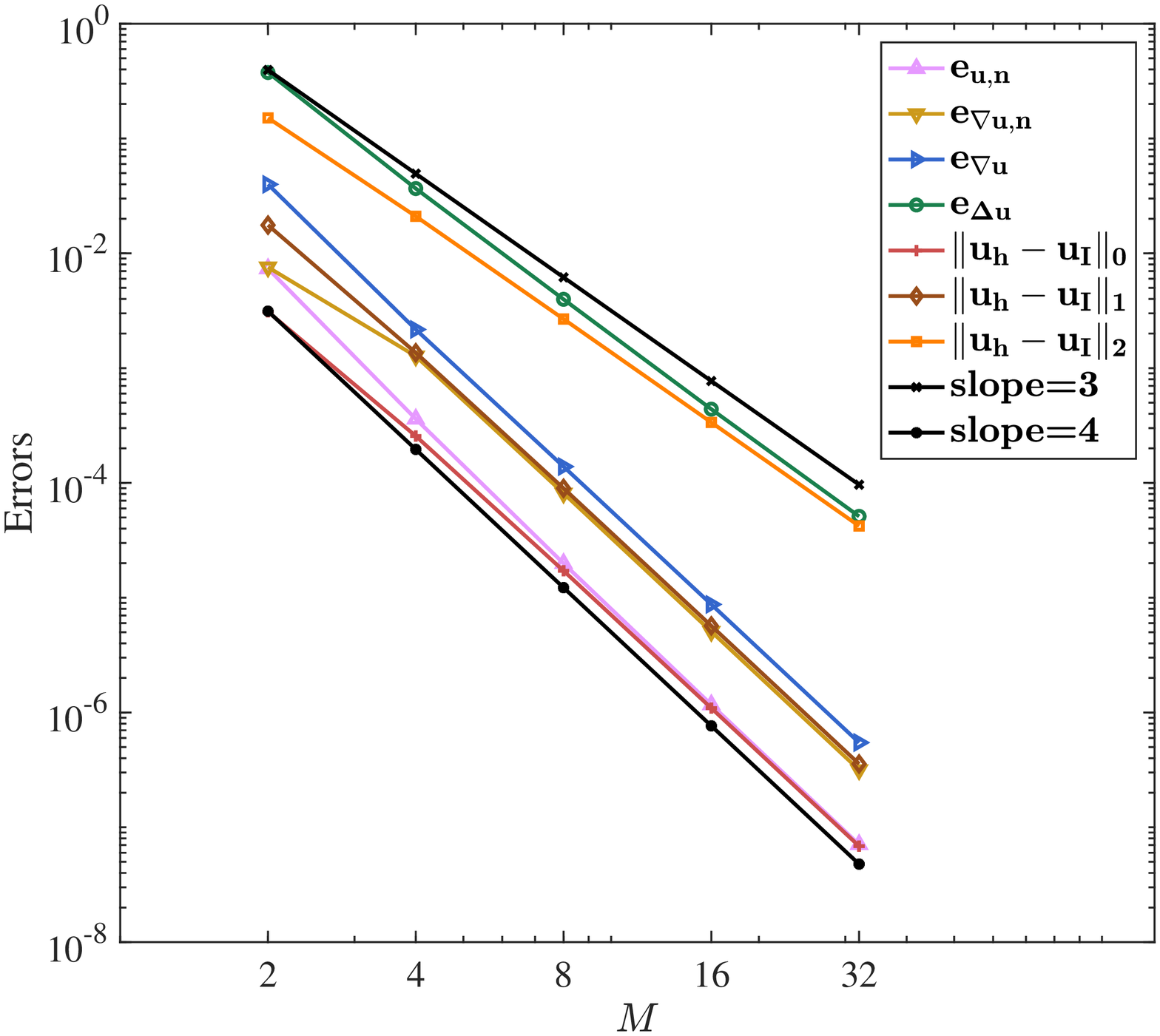}
\end{minipage}%
}%
\subfigure{
\begin{minipage}[t]{0.33\linewidth}
\centering
\includegraphics[width=1.6in,height=1.4in]{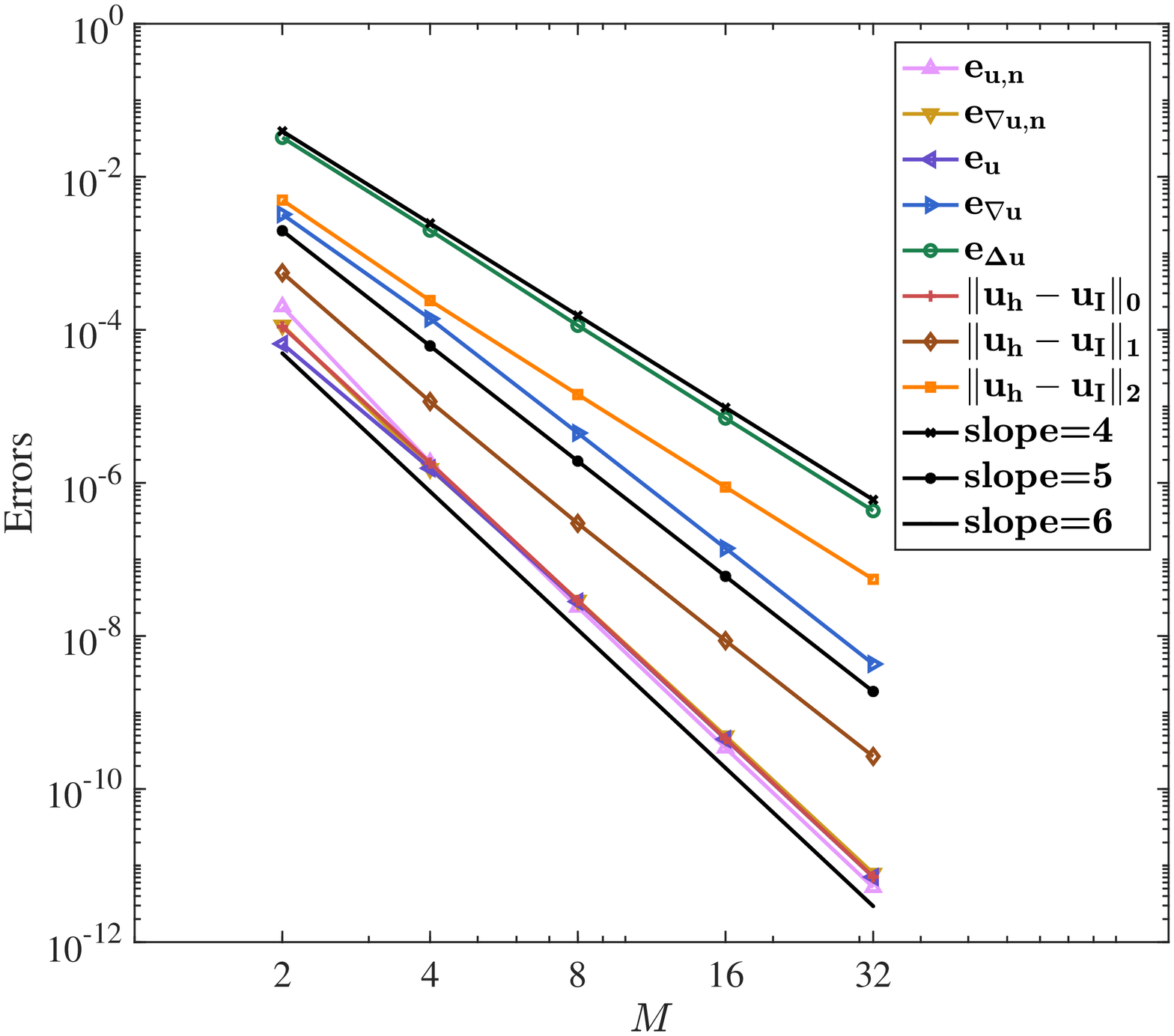}
\end{minipage}%
}%
\subfigure{
\begin{minipage}[t]{0.33\linewidth}
\centering
\includegraphics[width=1.6in,height=1.4in]{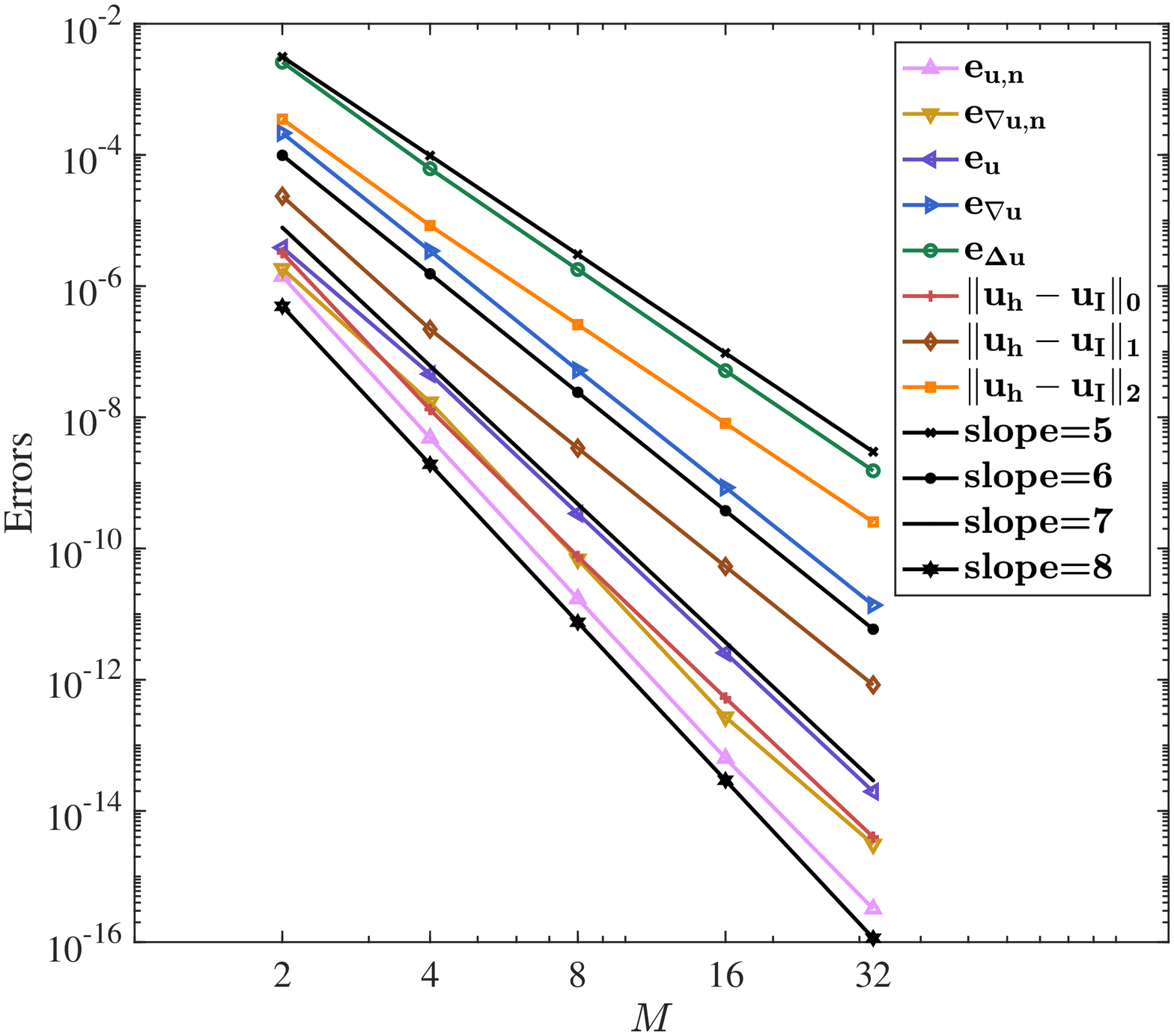}
\end{minipage}%
}%
\centering\vspace{-0.2in}
\caption{Error curves for Example 1 with $\alpha=1$, $\beta=(1,1)$, and $\gamma=1$. (Left: $k=3$, Middle: $k=4$, Right: $k=5$) }\label{const1_coefficient}
\end{figure}


To show the effect of the coefficients on the convergence rate, we further test different choice of  coefficients. Presented in Figures \ref{const2_coefficient}-\ref{const4_coefficient}  are error curves of $\|u_h-u_I\|_m, 0\le m\le 2$ in
three cases: $\alpha=1, \beta=(0,0), \gamma=1$,  $\alpha=1, \beta=(1,1), \gamma=0$, and $\alpha=1,\beta=(0,0),\gamma=0$.
We observe that the convergence   rate for the case
$\alpha=1, \beta=(1,1), \gamma=0$ is the same at that for the counterpart $\alpha=1, \beta=(1,1), \gamma=1$ in Figure \ref{const1_coefficient}.
However, in  the  case $\beta=(0,0)$, it seems that the convergence rate improves for $k=4,5$. To be more precise, we observe  a convergence
rate  $h^{k+2}$ for $\|u_h-u_I\|_1$ and $h^{k+1}$ for $\|u_h-u_I\|_2$ when $k=4,5$, one order higher than the  case  $\beta = (1,1)$.
In other words, it seems that the convection coefficient has effect on the superconvergence rate. 

\begin{figure}[H]
\centering
\subfigure{
\begin{minipage}[t]{0.33\linewidth}
\centering
\includegraphics[width=1.6in,height=1.4in]{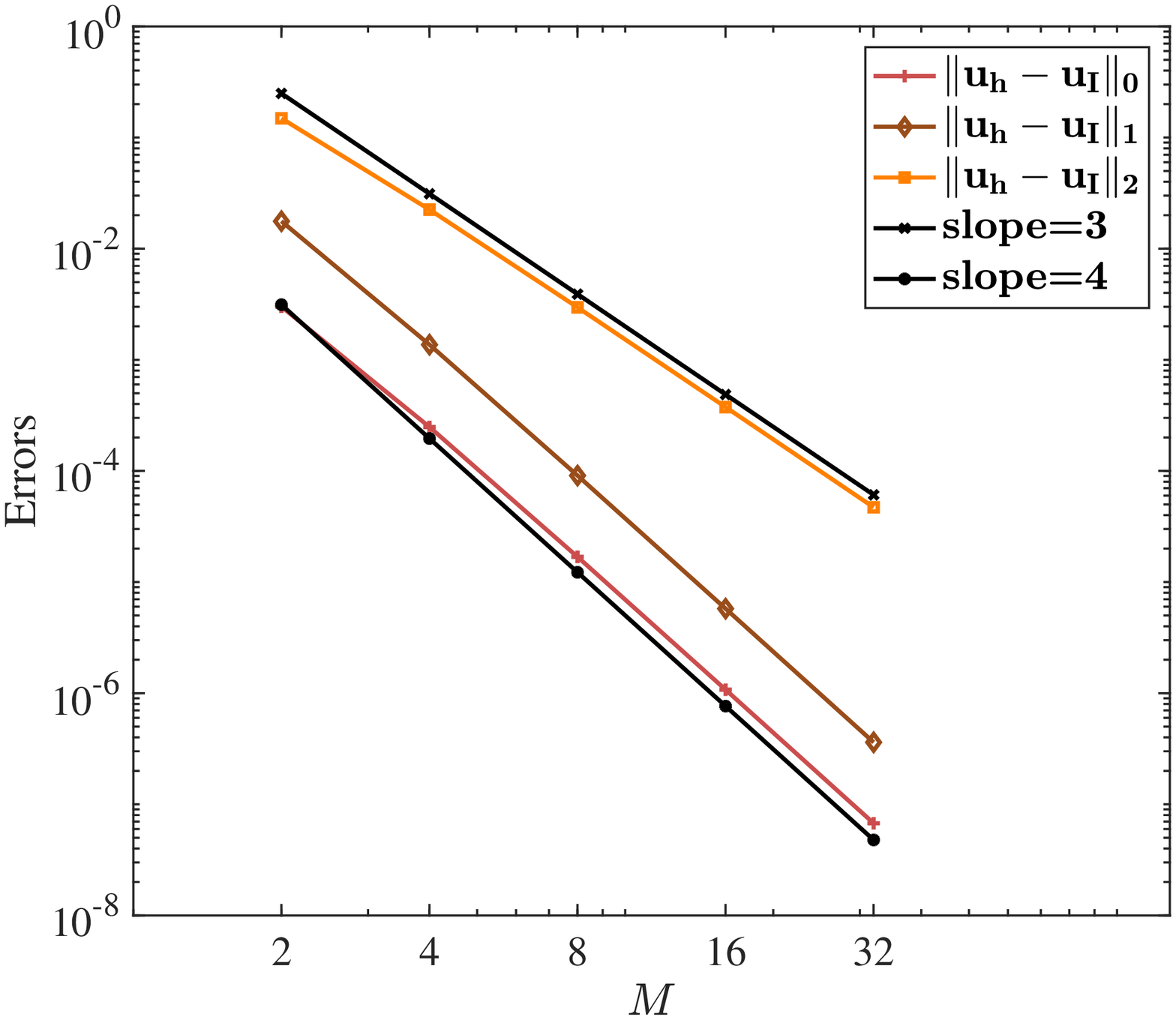}
\end{minipage}%
}%
\subfigure{
\begin{minipage}[t]{0.33\linewidth}
\centering
\includegraphics[width=1.6in,height=1.4in]{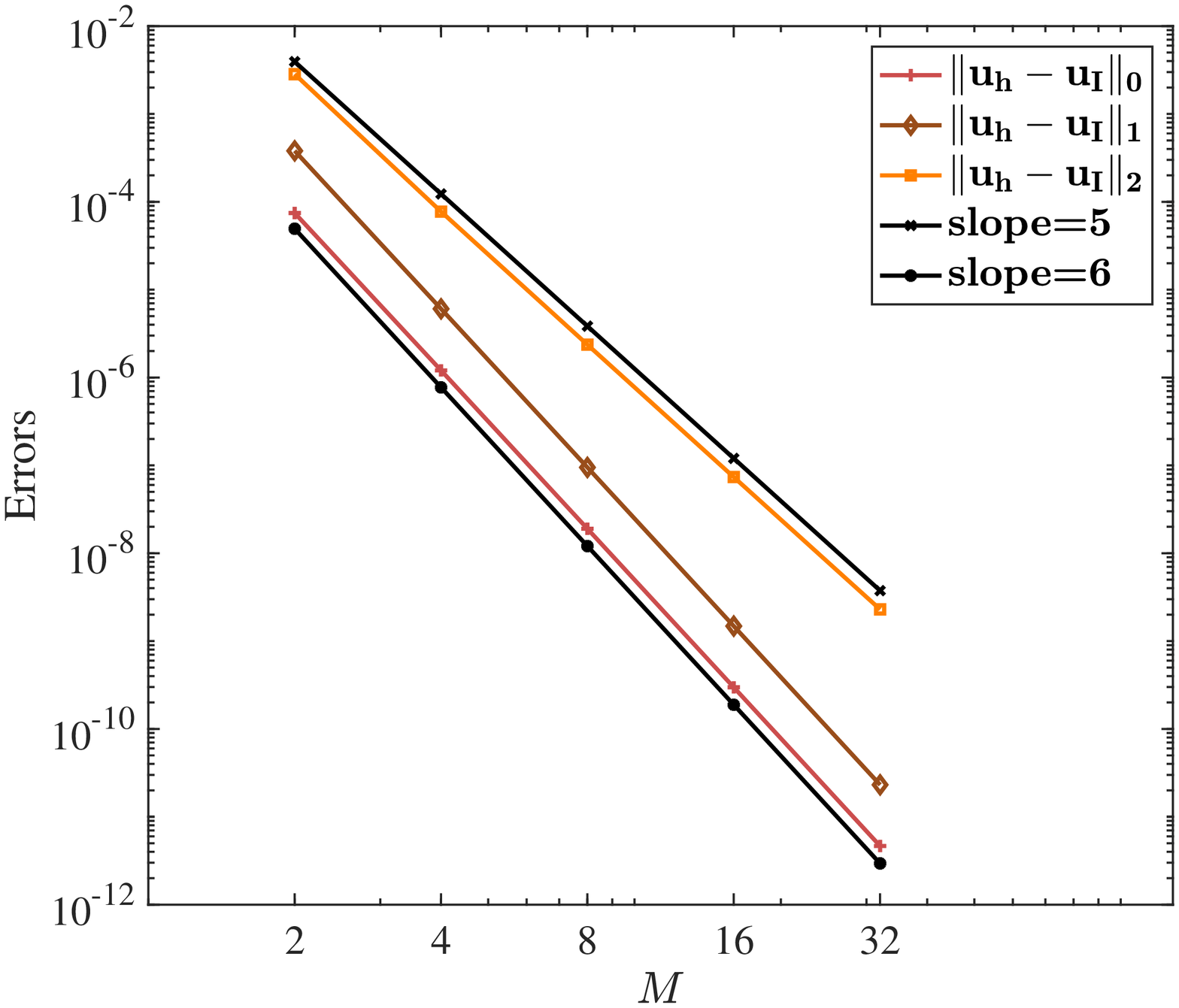}
\end{minipage}%
}%
\subfigure{
\begin{minipage}[t]{0.33\linewidth}
\centering
\includegraphics[width=1.6in,height=1.4in]{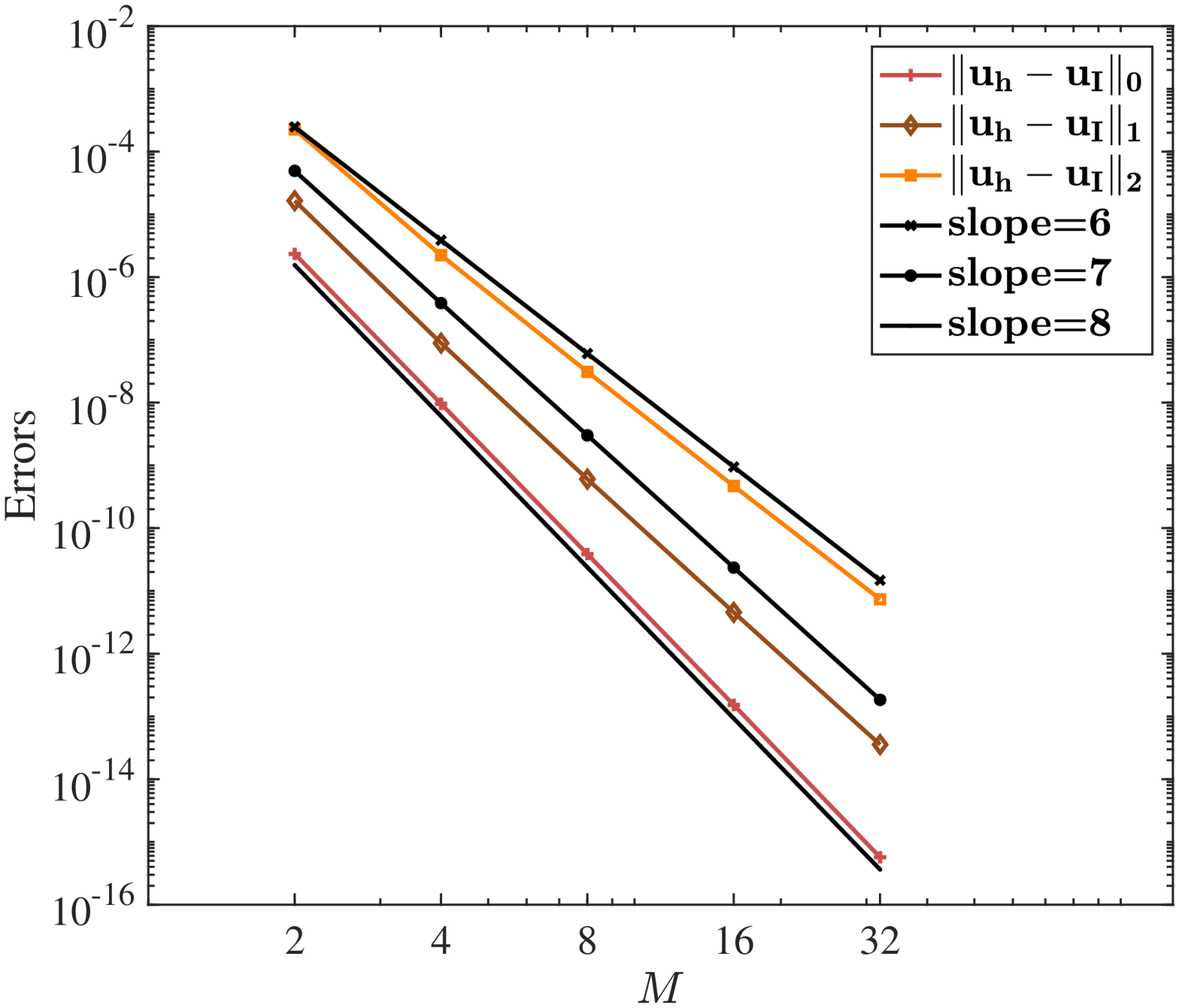}
\end{minipage}%
}%
\centering\vspace{-0.2in}
\caption{Error curves for Example 1 with $\alpha=1,$ $\beta=(0,0)$, and $\gamma=1$. (Left: $k=3$, Middle: $k=4$, Right: $k=5$) }\label{const2_coefficient}
\end{figure}

\begin{figure}[H]
\centering
\subfigure{
\begin{minipage}[t]{0.33\linewidth}
\centering
\includegraphics[width=1.6in,height=1.4in]{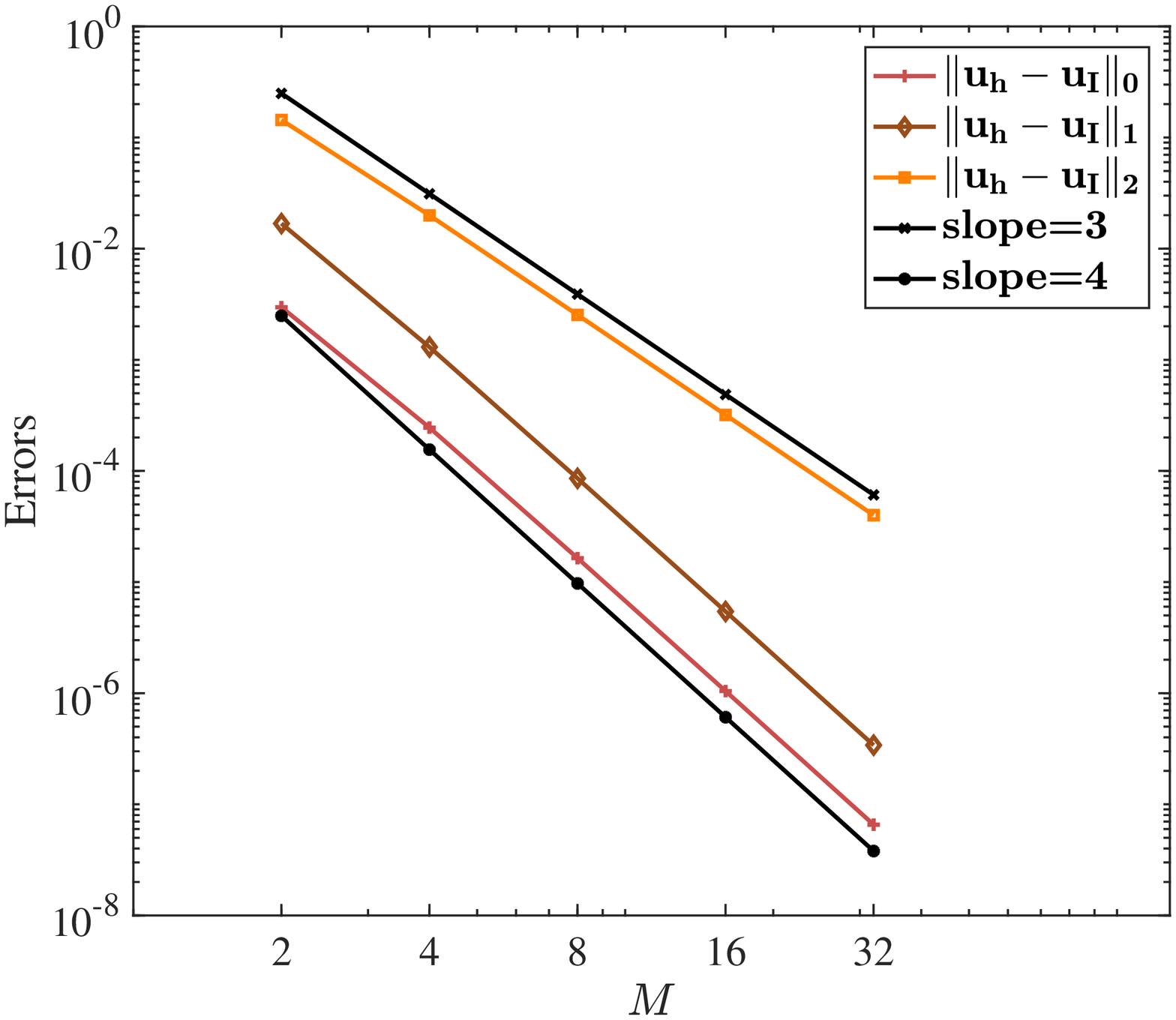}
\end{minipage}%
}%
\subfigure{
\begin{minipage}[t]{0.33\linewidth}
\centering
\includegraphics[width=1.6in,height=1.4in]{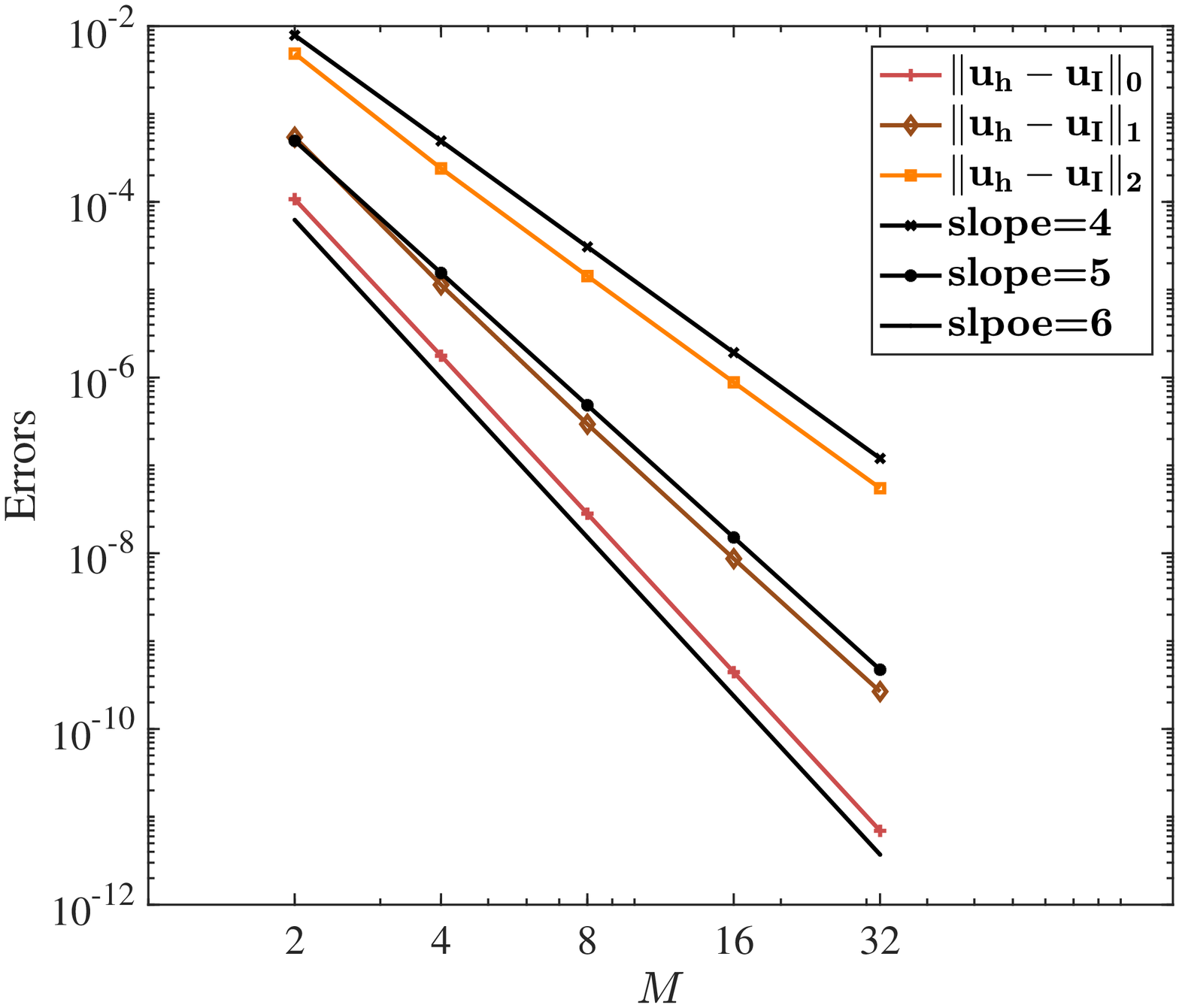}
\end{minipage}%
}%
\subfigure{
\begin{minipage}[t]{0.33\linewidth}
\centering
\includegraphics[width=1.6in,height=1.4in]{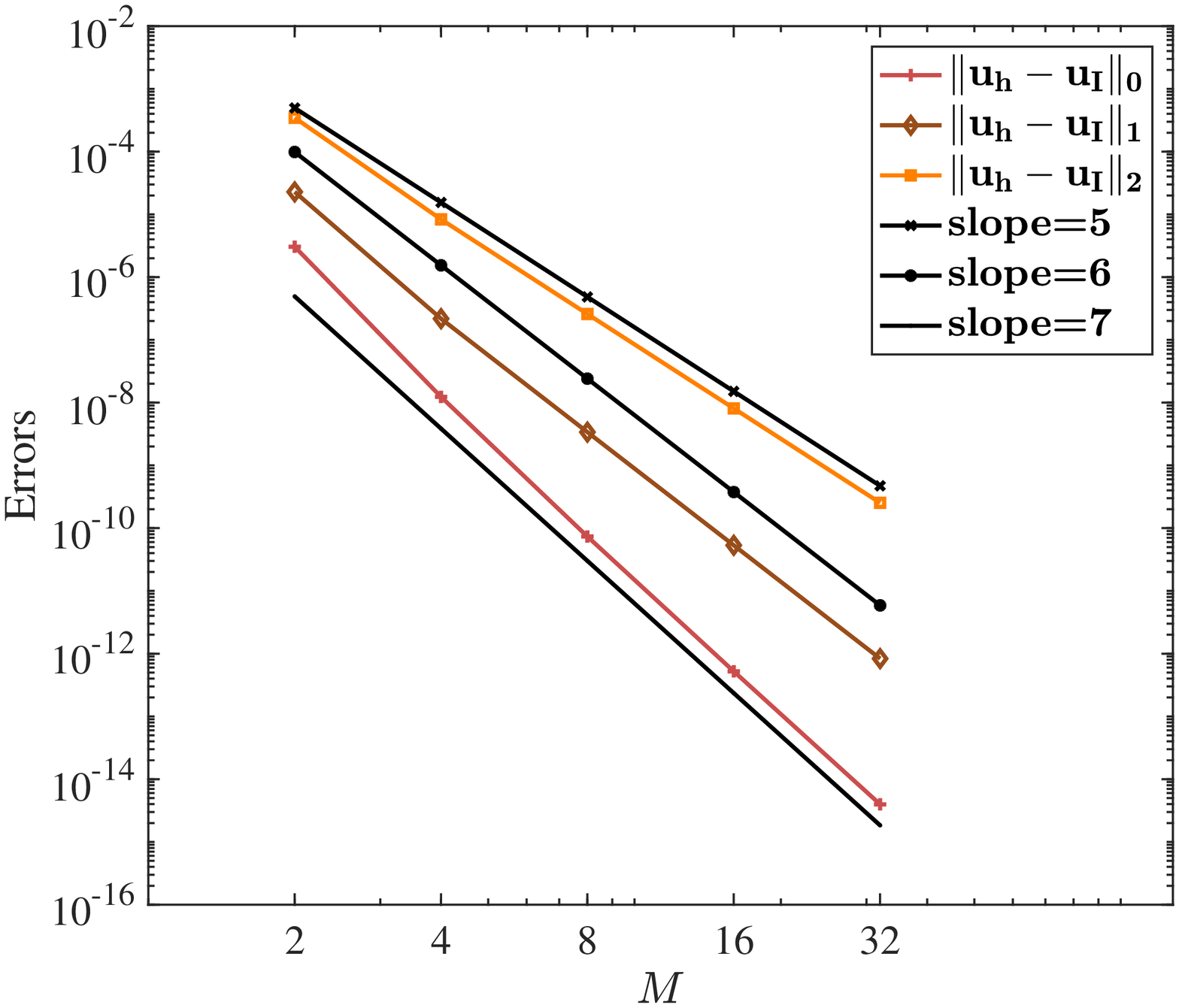}
\end{minipage}%
}%
\centering\vspace{-0.2in}
\caption{Error curves for Example 1 with $\alpha=1$, $\beta=(1,1)$, and $\gamma=0$. (Left: $k=3$, Middle: $k=4$, Right: $k=5$) }\label{const3_coefficient}
\end{figure}

We also present in  Figure \ref{const4_coefficient} the error curves for $e_{u,n}$ and $ e_{\nabla u,n}$
for  the  Poisson equation, i.e., $\alpha=1, \beta=(0,0),\gamma=0$. We observe a convergence rate of $h^{2k-2}$  for both $e_{u,n}, e_{\nabla u,n}$. Note that this superconvergence 
phenomenon for  two-dimensional case is different from that for 
the  one-dimensional case, where 
 both errors $e_{u,n}, e_{\nabla u,n}$ equal to zero (see, \cite{cao-jia-zhang}). 
\begin{figure}[H]
\centering
\subfigure{
\begin{minipage}[t]{0.33\linewidth}
\centering
\includegraphics[width=1.6in,height=1.4in]{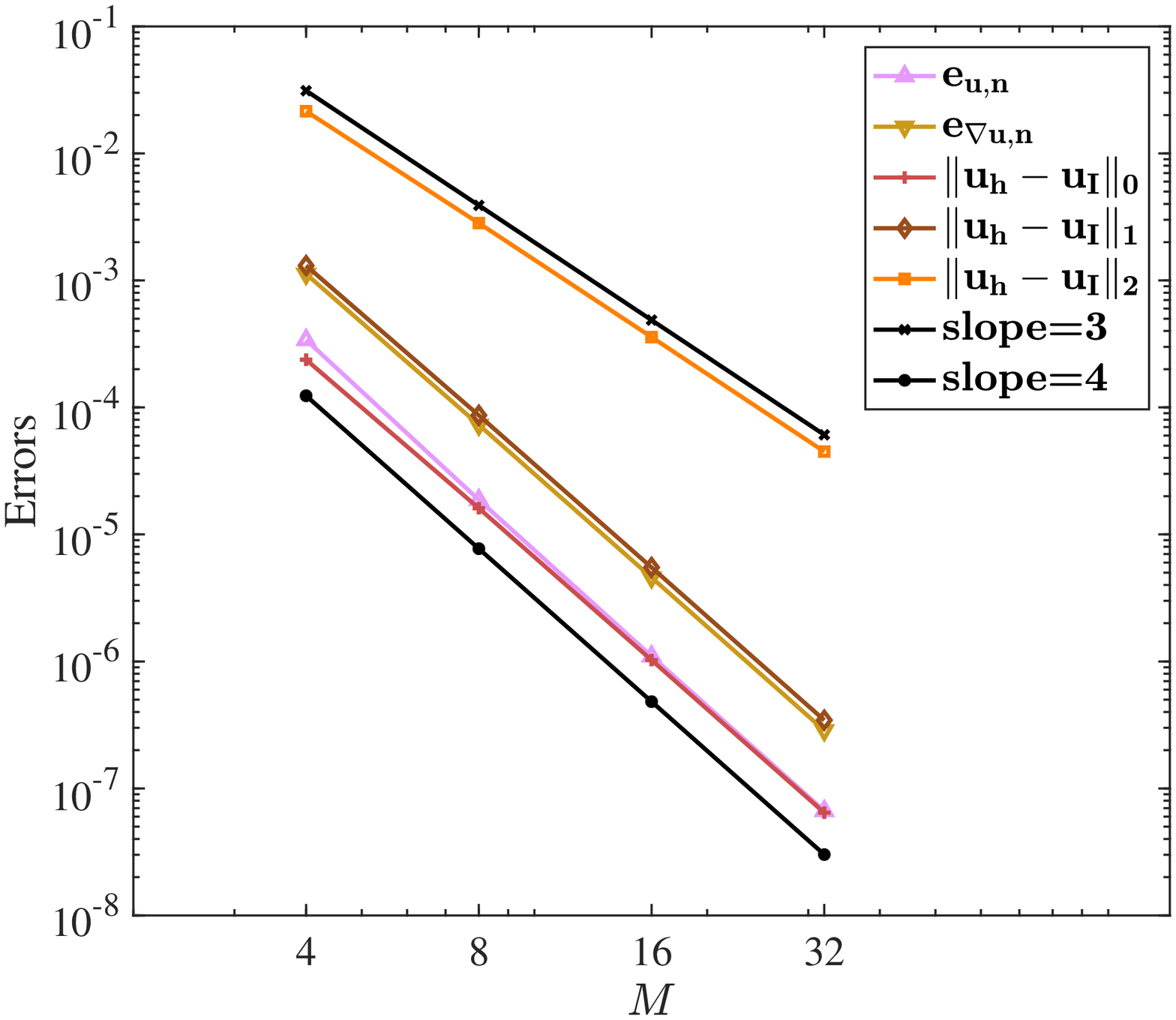}
\end{minipage}%
}%
\subfigure{
\begin{minipage}[t]{0.33\linewidth}
\centering
\includegraphics[width=1.6in,height=1.4in]{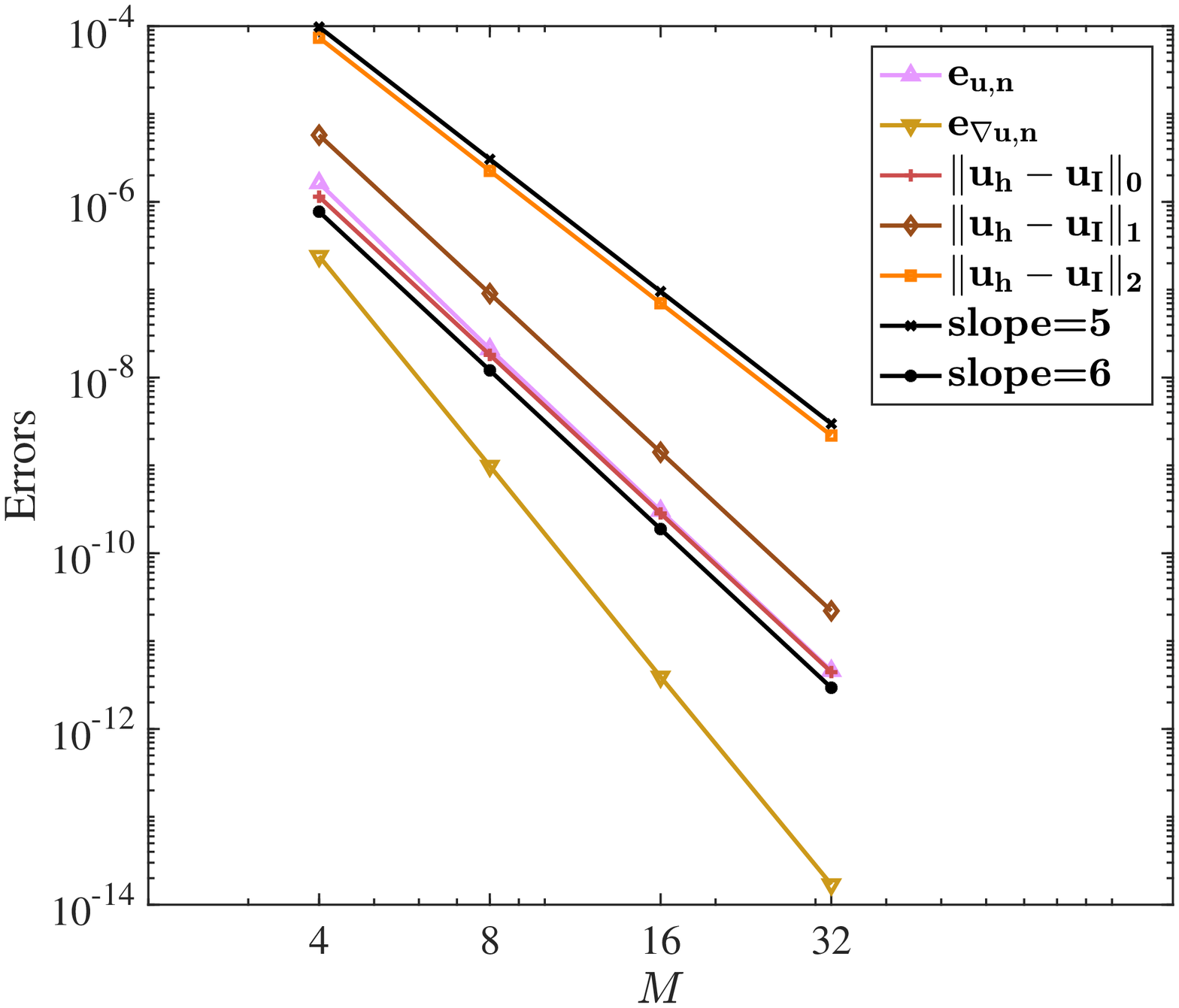}
\end{minipage}%
}%
\subfigure{
\begin{minipage}[t]{0.33\linewidth}
\centering
\includegraphics[width=1.6in,height=1.4in]{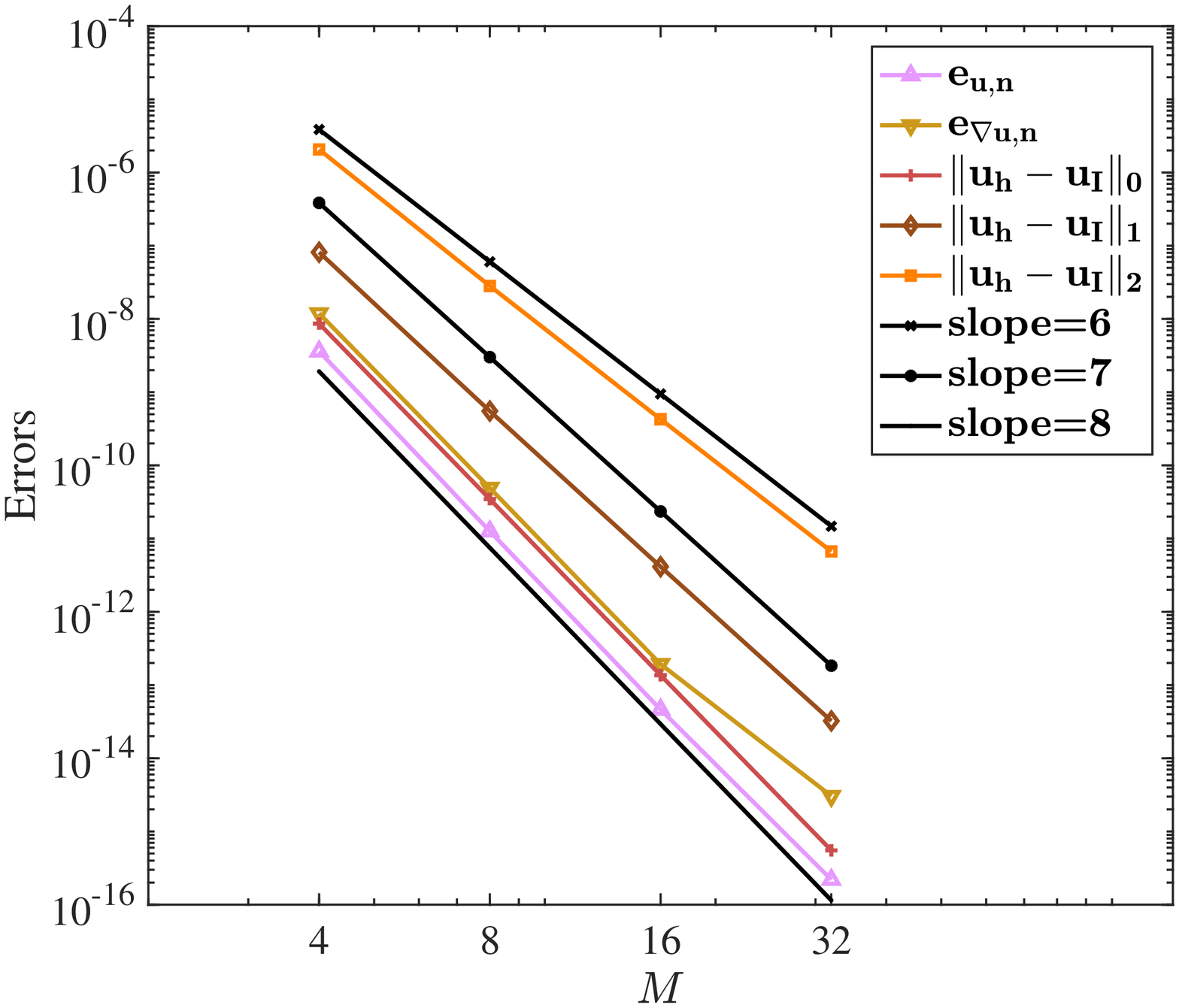}
\end{minipage}%
}%
\centering\vspace{-0.2in}
\caption{Error curves for Example 1 with $\alpha=1$, $\beta=(0,0)$, and $\gamma=0$. (Left: $k=3$, Middle: $k=4$, Right: $k=5$) }\label{const4_coefficient}
\end{figure}

{\it Example 2:} We consider the problem \eqref{con_laws}
  with the following variable coefficients: 
\begin{equation*}
\alpha(x,y) = e^{xy}, \quad \beta(x,y) = \left( x^2y, xy^2\right), \quad \gamma(x,y) = 2xy.
\end{equation*}
The right-hand side function $f(x,y)$ is chosen such that the exact solution is
$$u(x,y)=xy(1-e^{x-1})(1-e^{y-1}).$$

 The corresponding error curves for $k=3, 4, 5$ are presented in Figure \ref{variable_coefficient}.
We see that,  both convergence rates for $e_{u,n}$ and $e_{\nabla u,n}$ are  $h^{2k-2}$, and  convergence rates for $e_{u}$, $e_{\nabla u}$, $e_{\Delta u}$  are $h^{k+2}$, $h^{k+1}$, $h^k$, respectively. All these results again verify our theoretical findings in Theorems  \ref{theo:2}-\ref{theo:03}.
Just the same  as the constant coefficient case in Example 1, we observe a convergence rate $h^{\min\{k+2,2k-2\}}$ for $\|u_h-u_I\|_0$,
  $h^{k+1}$ for $\|u_h-u_I\|_1$,  and $h^k$ for $\|u_h-u_I\|_2$.
Again, all these  results  are consistent with  the error bounds established in Theorems  \ref{theo:02}-\ref{theo:03}.


\begin{figure}[H]
\centering
\subfigure{
\begin{minipage}[t]{0.33\linewidth}
\centering
\includegraphics[width=1.6in,height=1.4in]{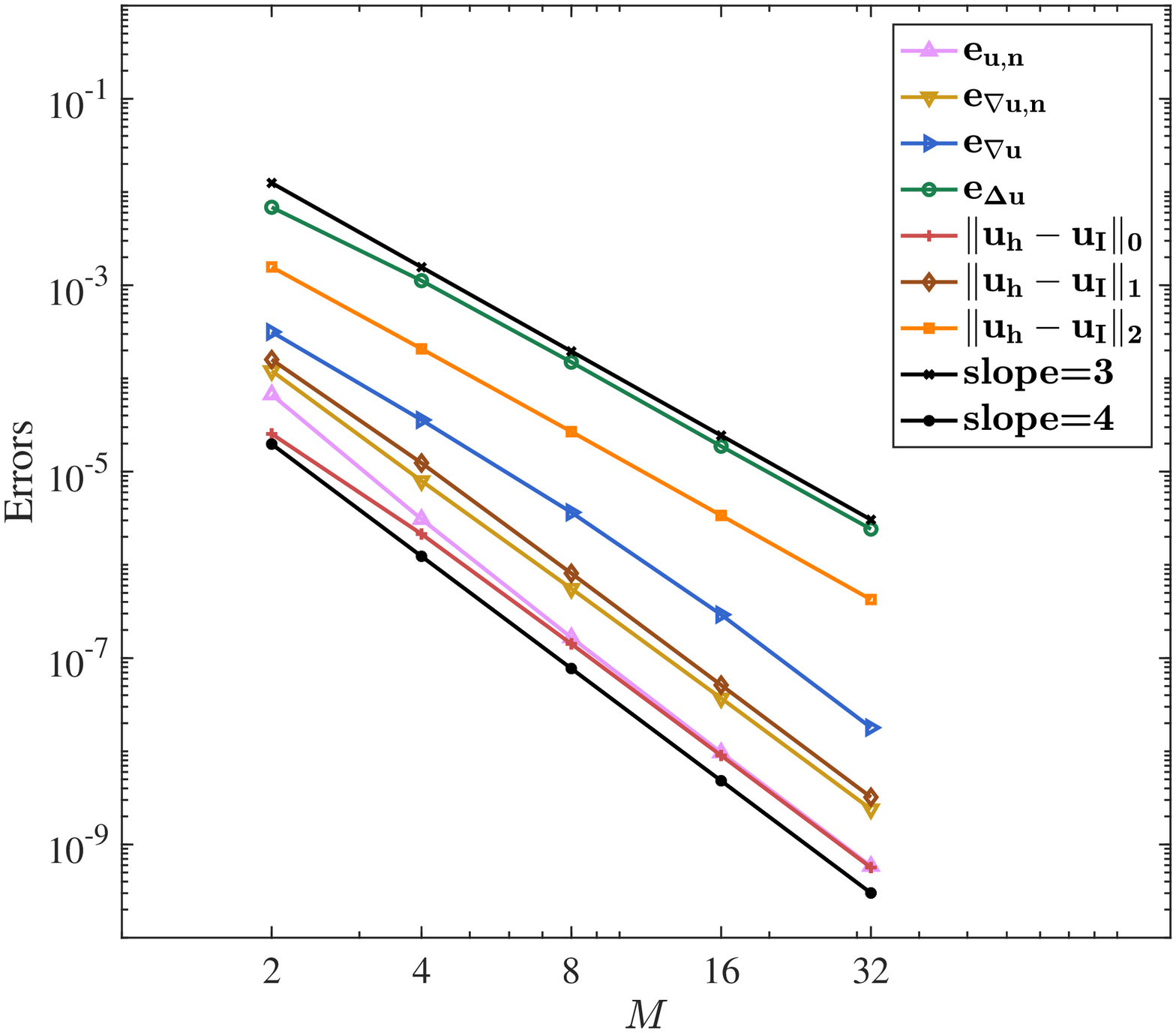}
\end{minipage}%
}%
\subfigure{
\begin{minipage}[t]{0.33\linewidth}
\centering
\includegraphics[width=1.6in,height=1.4in]{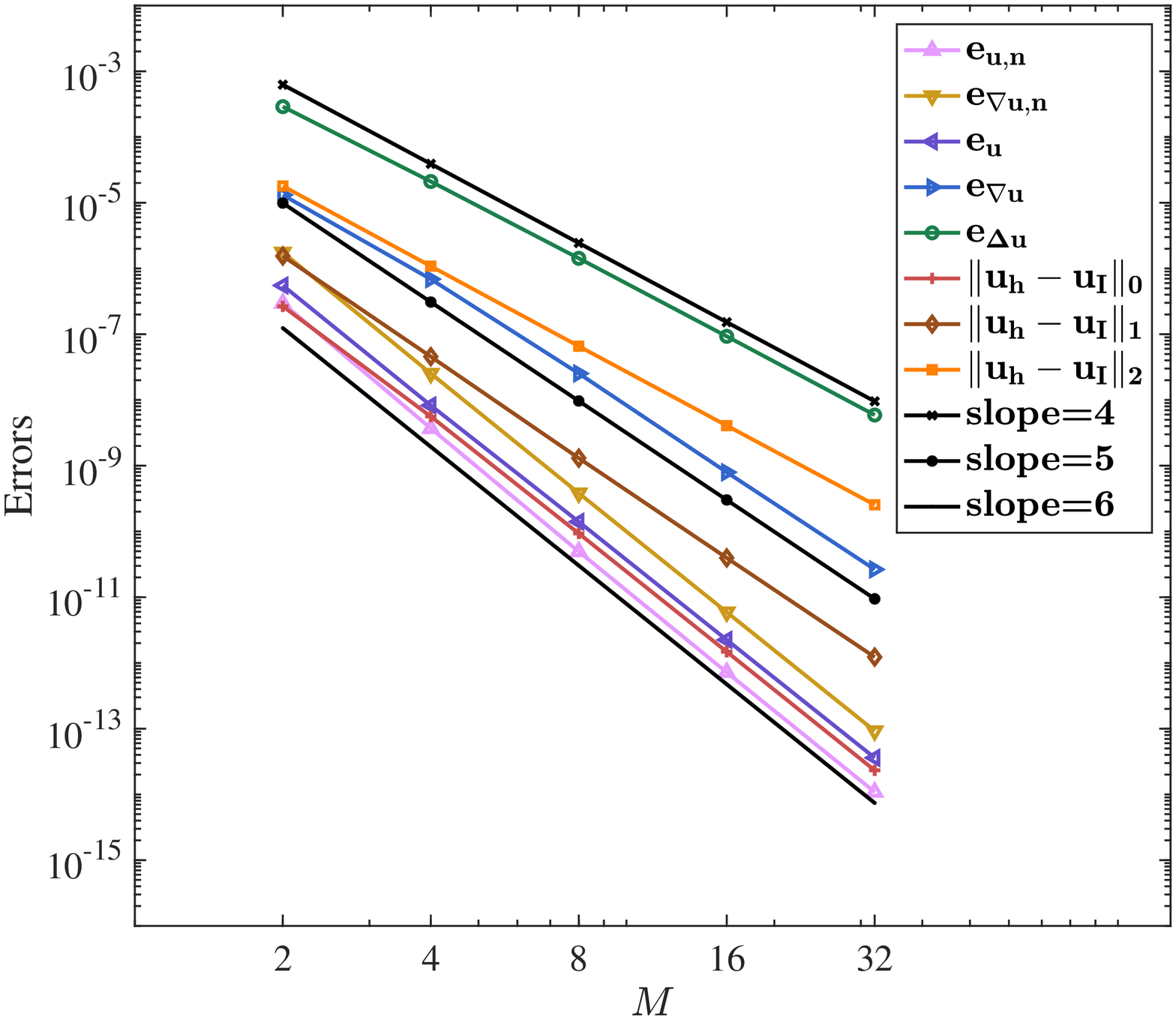}
\end{minipage}%
}%
\subfigure{
\begin{minipage}[t]{0.33\linewidth}
\centering
\includegraphics[width=1.6in,height=1.4in]{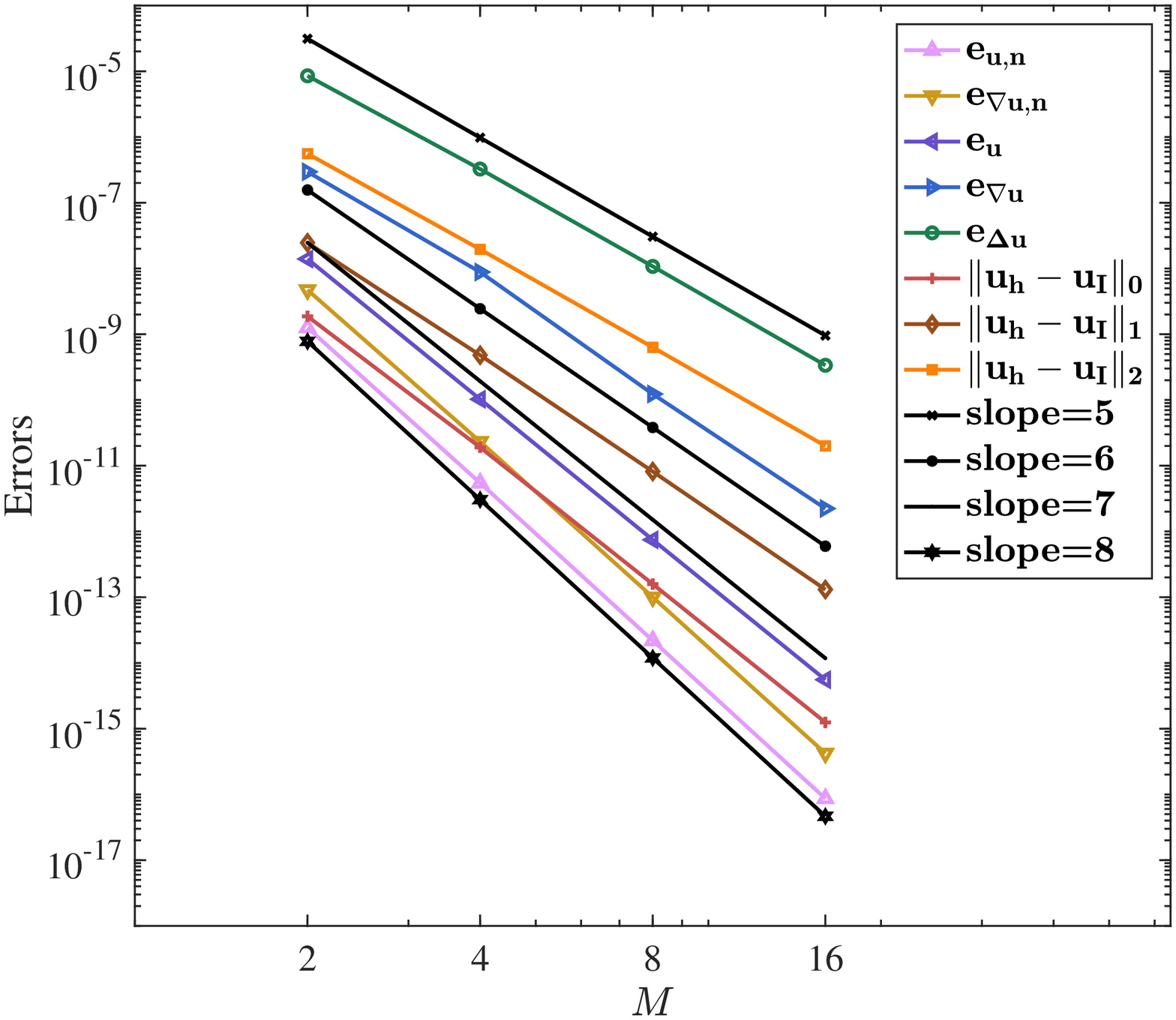}
\end{minipage}%
}%
\centering\vspace{-0.2in}
\caption{Error curves for Example 2. (Left: $k=3$, Middle: $k=4$, Right: $k=5$) }\label{variable_coefficient}
\end{figure}

\section{Conclusion}

 In this work, we have proposed a new $C^1$-$L^2$ Petrov-Galerkin method for convection-diffusion equations over rectangular meshes.
 The numerical scheme is designed to use the $C^1$-conforming ${\mathbb Q}_k$ element as our trial space and $L^2$  piecewise ${\mathbb Q}_{k-2}$ polynomials
 as our test space. We prove  that the designed numerical method is convergent with optimal rates in the $H^1, L^2, H^2$-norms, respectively.
 Furthermore, we have presented a unified   approach to study the superconvergence property of the Petrov-Galerkin method and
  establish the superconvergence results  including: 1) the function value and the first-order derivative 
  are superconvergent   with a rate of $h^{2k-2}$ at all mesh nodes; 2) the function value approximation is   superconvergent with rate
 $h^{k+2}$  at roots of $J^{-2,-2}_{k+1}(x)\otimes J^{-2,-2}_{k+1}(y)$; 3) the    first-order and second-order derivatives are superconvergent with rates of $h^{k+1}$ and $h^k$ along the Lobatto lines and Gauss lines, respectively; 4) the numerical solution $u_h$
  is superconvergent towards the special Jacobi projection $u_I$ of the exact solution in all $L^2,H^1,H^2$-norms.
  Numerical experiments demonstrate that all the established error
bounds are optimal.

We would like to point out that in principle it is straightforward to generalize the methodology we adopt in this paper to convection-diffusion equations with variable coefficients.
 However, it requires very tedious and lengthy arguments to carry on the argument, in a mathematically rigorous way. Our numerical results demonstrate that 
 the same convergence and superconvergence  results still hold true for convection-diffusion equations with variable coefficients.

\section{Appendix}

This section is dedicated to the construction of the correction function $w_h$ satisfying the conditions of Proposition 2. 

 In light of 
\eqref{decom:2} and the estimates  of $E^xE^yu$ in Lemma \ref{lemma:1}, we have 
\[
    a(u-u_I+w_h,\theta)=a(E^xu+E^yu+w_h,\theta)+O(h^{2k-2})\|\theta\|_0.
\]
  In other words, to achieve our superconvergence goal, we need to  construct correction functions $w_h^x,w_h^y\in V_h$ to separately correct the 
  two low-order errors $a(E^xu,\theta)$ and $a(E^xy,\theta)$
  such that
\[
    a(E^xu+w_h^x,\theta)=O(h^{k-1+l})\|\theta\|_0,\ \ a(E^yu+w_h^y,\theta)=O(h^{k-1+l})\|\theta\|_0
\]
  for some positive $l$.  

\subsection{Correction function for the error  $a(E^xu,\theta)$ }

We begin with the introduction of some operators $Q_h^y, Q_h^x$.
For any function $v(x,y)$,  we define a special operator $Q_h^y$ along the $y$-direction as follows:
$Q_h^yv|_{\tau_j^y}\in\mathbb P_{k}(\tau_j^y)$ satisfies
\begin{eqnarray*}
  && (Q_h^yv-v)(\cdot,y_{j})=\partial_y(Q_h^yv-v)(\cdot,y_{j})=0,\ j\in\bZ_N,\\
  &&\int_{y_{j-1}}^{y_j}(Q_h^yv-v)(\cdot,y)\theta dy=0,\ \ \forall\theta \in\mathbb P_{k-4}(\tau_j^y),\ k\ge 4.
\end{eqnarray*}
   Note that the operator $Q_h^y$ is actually the one dimensional truncated Jacobi expansion along the $y$-direction  while
   the other variable is fixed.  Consequently,
\begin{equation}
   \|v-Q_h^yv\|_{p,\infty,\tau_j^y}=\|E^yv\|_{p,\infty,\tau_j^y}\lesssim h^{k+1-p}\|v\|_{p,\infty,\tau_j^y},\ \ \forall p\le k+1. 
\end{equation}
  Similarly we can define the operator $Q_h^x$ along the $x$-direction.  

In light of \eqref{Ex}, we have   
  \begin{equation}\label{decom:1}
   E^xu|_{\tau_{ij}}=\sum_{p=k+1}^{\infty} c_{i,p}(y)J_{j,p}^{-2,-2}(x).
\end{equation}
 Define
\begin{equation}\label{vip}
    u_0(x,y)|_{\tau_{ij}}:=
   c_{i,k+1}(y)J^{-2,-2}_{i,k+1}(x)+c_{i,k+2}(y)J^{-2,-2}_{i,k+2}(x),
\end{equation}
 By the orthogonality  of Jacobi polynomials, i.e.,  $J^{-2,-2}_{n+1}\bot\mathbb P_{n-4}, \partial_sJ^{-2,-2}_{n+1}\mathbb \bot {\mathbb P}_{n-3}$ (see \cite{cao-jia-zhang}), we have 
\begin{eqnarray}\label{EExu:1}
\begin{split}
   a(E^xu,\theta)&=(-\alpha E^xu_{yy}+\beta_1\partial_xE^xu+\beta_2 E^xu_y+\gamma E^xu,\theta)&\\
     &= (\beta_1\partial_xu_0,\theta)+ ((-\alpha \partial_{yy}^2+\beta_2\partial_y+\gamma)u_0,\theta). &
\end{split}
\end{eqnarray}
   Let
\begin{eqnarray}\label{lam:1}
   \lambda_1(y)=Q_h^yc_{i,k+1}(y),\ \ \lambda_2(y)=Q_h^yc_{i,k+2}(y),
\end{eqnarray}
  and define
 \begin{equation}\label{w0}
    w_0(x,y)|_{\tau_i^x}=\lambda_1(y)J_{i,k+1}^{-2,-2}(x)+\lambda_2(y)J_{i,k+2}^{-2,-2}(x),\ \ w_{-1}(x,y)=0.
 \end{equation}
  For all $l\in\bZ_{k-2}$,  we define a series of functions $w_l\in V_h$ as follows:
  \begin{eqnarray}\label{wwx}
  &&\alpha(\partial_{xx}w_{l},\theta)=(\beta_1\partial_x w_{l-1}+ (\beta_2\partial_y-\alpha\partial^2_{yy}+\gamma)w_{l-2},\theta ),\ \ \forall\theta\in{\mathcal S}^x_h,\\\label{wwx:1}
  &&\partial_{x}w_{l}(x_i,y)=0,\ \ w_{l}(a,y)=0,\ \  w_{l}(x,c)=w_{l}(x,d)=0,\ \forall (x,y)\in\Omega.
\end{eqnarray}
  where
\begin{equation}\label{sh}
   {\mathcal S}^x_h:=\{\theta(x,y): \theta|_{\tau}\in(\mathbb P_{k-2}(x)\setminus \mathbb P_0(x))\times \mathbb P_{k-2}(y),\ \forall\tau\in{\mathcal T}_h\}.
\end{equation}

\begin{lemma}\label{lemma:000} Define 
\[
    W^y_h:=\{v(y)\in C^1([c,d]):  v|_{\tau_j^y}\in\mathbb P_{k},\ v(c)=v(d)=0,\  j\in\bZ_N\}. 
\]
  Given any smooth function $g$,  assume that $v(y)\in W^y_h$ is the solution of the following problem: 
  \begin{equation}\label{eqq:001}
      \int_{c}^d v(y)\theta(y) dy= \int_{c}^d g(y)\theta(y) dy,\ \ \forall \theta\in \mathbb P_{k-2}(\tau_j^y). 
  \end{equation}
    Then $v(y)$ is well defined. Moreover, there holds 
\begin{equation}\label{eqq:002}
    \|\partial_y^nv\|_{0,\infty,[c,d]}\lesssim \|\partial_y^{n}g\|_{0,\infty,[c,d]},\ \ \forall n\le k. 
\end{equation}
 \end{lemma}
 \begin{proof}  To prove the uniqueness of $v$, we only need to show that the right hand side function $g=0$ can yield a unique zero solution, i.e., $v=0$. 
 To this end, we choose $\theta=\partial_{yy}v$ in \eqref{eqq:001} and use the integration by parts to obtain 
\[
   \int_{c}^d (\partial_y v)^2 dy=0, 
\]
   which yields, together with the homogenous boundary condition,
\[
   \partial_y v=0,\ \ v=0.
\]
 Consequently, $v(y)$ is well defined. 
 
 To estimate \eqref{eqq:002}, we define a special function $ R_{h}g\in W_h^y$ of $g$ as follows: 
\[
     R_{h}g|_{\tau_j^y}:=Q_h^y g,\ \ j\neq 1, N,  \ \   R_{h}g(y_{j,m})=g(y_{j,m}),\ \ j=1,N, m\in\bZ_{k-3}, 
\]
  where $y_{j,m}$ can be chosen as any interpolation points.   By the approximation theory, we have for all $n\le k$, 
\[
    \|g- R_{h}g\|_{0,p,\tau_j^y}\lesssim\left\{
    \begin{array}{lll}
     h^{n+1}\|g\|_{n+1,p,\tau_j^y},& j\neq 1,N,\\
     h^{n}\|g\|_{n,p,\tau_j^y}, & j=1,N.
     \end{array}
     \right.
\]
  Then we choose $\theta=\partial_{yy}(v- R_{h}g)$ in \eqref{eqq:001} to derive 
 \begin{eqnarray*}
 \|\partial_y(v- R_{h}g)\|_1^2&= &|(v- R_{h}g, \partial_{yy}(v- R_{h}g))|
= | (g- R_{h}g, \partial_{yy}(v- R_{h}g))|\\
&=&\sum_{j=2}^{N-1}| (\partial_y(g- R_{h}g), \partial_{y}(v- R_{h}g))| +h^{n+1}\|g\|_{n,\infty,\tau_0}\|v- R_{h}g\|_{2,\infty,\tau_0}\\
&\lesssim & h^{n-\frac 12}\|g\|_{n,\infty}\|v- R_{h}g\|_1,
 \end{eqnarray*}
  where $\tau_0=\tau_1^y\cup\tau_N^y$, and in the last step, we have used the inverse inequality $ \|v\|_{2,\infty}\lesssim h^{-\frac{3}{2}} \|v\|_{1}$ for any finite element function $v$. 
  Again, by the inverse inequality, we have 
 \[
    \|v- R_{h}g\|_n\lesssim h^{1-n}\|v- R_{h}g\|_{1}\lesssim h^{\frac 12}\|g\|_{n,\infty}, \ \ \forall n\le k, 
 \]
    and thus
 \[
     \|v- R_{h}g\|_{n,\infty}\lesssim h^{-\frac 12} \|v- R_{h}g\|_n\lesssim \|g\|_{n,\infty}. 
  \]
  Consequently, 
\[
   \|v\|_{n,\infty}\lesssim \|v- R_{h}g\|_{n,\infty}+\| R_{h}g\|_{n,\infty}\lesssim \|g\|_{n,\infty}. 
\]
 This finishes the proof of \eqref{eqq:002}.  $\Box$
 \end{proof}

\begin{lemma}\label{lemma:3}
  Assume that $u$ has the Jacobi expansion \eqref{uu} in each $\tau_{ij}$, and $\lambda_1,\lambda_2$ are defined in \eqref{lam:1}.
  Then the correction functions $w_l, 1\le l\le k-2$  in \eqref{wwx}-\eqref{wwx:1} are well defined.
  Furthermore, there exist functions $\mu_{l,p}(y)\in W^y_h$ such that 
\begin{equation}\label{eqq:7}
   \partial_xw_l|_{\tau_i^x}=h^{l-1}\sum_{p=k-l}^{k-1} \mu_{l,p}(y)\phi_{i,p}(x),\ \|\partial^n_y\mu_{l,p}\|_{0,\infty}\lesssim
   \|\lambda_1\|_{n+l,\infty}+\|\lambda_2\|_{n+l,\infty}. 
\end{equation}
  Consequently, if $u\in W^{k+1+l+n,\infty}$ with $n$ being some positive integer, then
\begin{equation}\label{opti:wl}
   \|\partial_y^n \partial_xw_l\|_{0,\infty}\lesssim h^{k+l}\|u\|_{k+l+n+1,\infty},\ \|\partial_y^n w_l\|_{0,\infty}\lesssim h^{\min(k+l+1,2k-2)}\|u\|_{k+l+n+1,\infty}.
\end{equation}
  Here $\phi_{i,n}(x)$ denotes the Lobatto polynomial of degree $n$ on $\tau_i^x$, i.e.,
\[
    \phi_{i,n}(x)=\phi_{n}(\frac{2x-x_i-x_{i-1}}{h_i})=\phi_n(s),\ \ s\in [-1,1],\ x\in(x_{i-1},x_i).
\]

\end{lemma}
\begin{proof}
   Note that for any $w_l\in V_h$, we have
   $\partial^2_{xx}w_l\in\mathbb P_{k-2}(\tau_i^x)$. Soppose
 \[
    \partial_{xx}w_l|_{\tau_i^x}=\sum_{p=0}^{k-2}c_{l,p}(y)L_{i,p}(x). 
 \]
   Recalling the definition of $w_l$ in \eqref{wwx}, we easily obtain that $c_{l,p}(y)$ is the solution of \eqref{eqq:001} with the right hand side function 
 \[
     g(y)=\frac{2p+1}{h_i^x\alpha}\int_{\tau_i^x}(\beta_1\partial_x w_{l-1}+ (\beta_2\partial_y-\alpha\partial^2_{yy}+\gamma)w_{l-2})(x,y)L_{i,p}(x) dx. 
 \]
   In other words, $c_{l,p}(y)$ is uniquely determined, and thus
   $\partial_{xx}w_l$ is well-defined.   Then the homogenous boundary condition  in \eqref{wwx:1} indicates a unique function $w_l$ from $\partial_{xx}w_l$. 
     Consequently $w_l$ is uniquely  defined.

  We  prove \eqref{eqq:7} by the method of mathematical induction.  We first show \eqref{eqq:7} is valid for $l=1$. Note that $w_0(x_{i},y)=0, y\in [c,d], i\in\bZ_M$ and
  $\partial_x\theta|_{\tau_i^x}\in {\mathbb P}_{k-3}(x)$ for any $\theta\in W_h$.
  By \eqref{wwx} and the integration by parts,  we have
\begin{eqnarray*}
   \alpha(\partial_xw_1,\partial_x\theta)&=&\beta_1(w_0,\partial_x\theta)=\beta_1\sum_{i=1}^M\int_{c}^d \lambda_1(y)dy\int_{x_{i-1}}^{x_i}J^{-2,-2}_{i,k+1}(x)\partial_x\theta(x,y) dx\\
     &=&-\frac{4\beta_1(k-1)(k-2)}{2k-1}\sum_{i=1}^M\int_{c}^d \lambda_1(y)dy\int_{x_{i-1}}^{x_i}\phi_{i,k-1}(x)\partial_x\theta(x,y) dx.
\end{eqnarray*}
  Here  in the last step, we have used   $J_{n+1}^{-2,-2}(s)=\frac{4(n-1)(n-2)}{2n-1}(\phi_{n+1}-\phi_{n-1})(s)$ 
 and the fact  $\phi_{k+1}\bot \mathbb P_{k-2}$.  Since $\lambda_1(c)=\lambda_1(d)=0$
  and the above equation holds for any $\theta$, then 
\[
   \partial_xw_{1}|_{\tau_i^x}=-\frac{4\beta_1(k-1)(k-2)}{2k-1}\lambda_1(y)\phi_{i,k-1}(x),
\]
   and thus \eqref{eqq:7} holds true for $l=1$ with
\[
  \mu_{1,k-1}(y)=-\frac{4\beta_1(k-1)(k-2)}{2k-1}\lambda_1(y),  \ \ \|\partial_y^n\mu_{1,k-1}\|_{0,\infty}\lesssim \|\partial_y^{n}\lambda_1\|_{0,\infty}\lesssim\|\lambda_1\|_{n,\infty}. 
\]

   Now we suppose \eqref{eqq:7} is valid for all $l$ and prove it holds true for $l+1$  with $l\le k-3, k\ge 4$.
 By the induction assumption and the orthogonality  of the Lobatto polynomials, we get
\[
   \partial_xw_l|_{\tau_i^x}=h^{l-1}\sum_{p=k-l}^{k-1} \mu_{l,p}(y)\phi_{i,p}(x)\bot \mathbb P_0(\tau_i^x),\ \ \forall l\le k-3,
\]
  and thus
\begin{eqnarray*}\label{wwl}
\begin{split}
   w_l(x,y)&=\int_{a}^x\partial_xw_l(x,y)dx
   =h^{l-1}\sum_{p=k-l}^{k-1} \mu_{l,p}\int_{x_{i-1}}^x\phi_{i,p}(x)dx&\\
   &=h^{l}\sum_{p=k-l}^{k-1}\frac{\mu_{l,p}}{2p-1} (\phi_{i,p+1}-\phi_{i,p-1})(x)
   =h^{l}\sum_{p=k-l-1}^k(\frac{\mu_{l,p-1}}{2p-3}-\frac{\mu_{l,p+1}}{2p+1}) \phi_{i,p}(x),&
\end{split}
\end{eqnarray*}
  where we use the notation $\mu_{l,p}=0$ for all $ p=k-l-1,k-l-2,k+1$, and
  in the forth step, we have used \eqref{lob}.
  Similarly, there holds
\begin{eqnarray*}
  w_{l-1}(x,y)|_{\tau_i^x}= h^{l-1}\sum_{p=k-l}^k(\frac{\mu_{l-1,p-1}}{2p-3}-\frac{\mu_{l-1,p+1}}{2p+1})\phi_{i,p}(x)\bot\mathbb P_0(\tau_i^x),\ \ \forall l\le k-3.
\end{eqnarray*}
  By defining
\[
   \partial_x^{-1}v(x,\cdot)=\int_{a}^x v(x,\cdot)dx,
\]
  and using the fact that $w_{l-1}\bot {\mathbb P}_0(\tau_i^x)$ and \eqref{lob}, we get
\begin{eqnarray*}
   \partial_x^{-1}w_{l-1}(x,y)&=&
   h^{l-1}\sum_{p=k-l}^k(\frac{\mu_{l-1,p-1}}{2p-3}-\frac{\mu_{l-1,p+1}}{2p+1})\int_{x_{i-1}}^x\phi_{i,p}(x)dx\\
  &=&h^{l}\sum_{p=k-l-1}^k \bar\mu_{l-1,p}\phi_{i,p}(x),
\end{eqnarray*}
  where
\begin{equation}
   \bar\mu_{l-1,p}=(\frac{\mu_{l-1,p-2}}{2p-5}-\frac{\mu_{l-1,p}}{2p-1})\frac{1}{2p-3} -(\frac{\mu_{l-1,p}}{2p-1}-\frac{\mu_{l-1,p+2}}{2p+3})\frac{1}{2p+1},
\end{equation}
   with
$
   \mu_{l-1,p}=0,\ \forall p\le k-l,\ {\rm or}\ p\ge k+1.$
   Note that
\begin{equation}\label{wi}
   w_l(x_i,\cdot)=\partial_x^{-1}w_{l-1}(x_i,\cdot)=0,\ \ \forall i\in\bZ_M, l\le k-3.
\end{equation}
  By  \eqref{wwx:1} and the integration by parts, 
\begin{eqnarray*}
   \alpha(\partial_xw_{l+1},\partial_x\theta)&=&(\beta_1 w_{l}+ (\beta_2\partial_y-\alpha\partial^2_{yy}+\gamma)\partial^{-1}_xw_{l-1},\partial_x\theta)\\
      &=&h^l\sum_{\tau_{i,j}\in{\mathcal T}_h}\sum_{p=k-l-1}^{k-1}\left(\bar c_{i,p}\phi_{i,p},\partial_x\theta\right)_{\tau_{i,j}}.
\end{eqnarray*}
  Here $(u,v)_{\tau}=\int_{\tau} (uv)(x,y) dxdy$,  and 
 \[
    \bar c_{i,p}=\beta_1\big(\frac{\mu_{l,p-1}}{2p-3}-\frac{\mu_{l,p+1}}{2p+1}\big)+(\beta_2\partial_y-\alpha\partial_{yy}+\gamma)\bar \mu_{l-1,p}. 
 \]
   Consequently, 
  \begin{eqnarray*}
   \partial_xw_{l+1}|_{\tau_{ij}}
   =h^l\sum_{p=k-l-1}^{k-1}\mu_{l+1,p}(y)\phi_{i,p}(x)
\end{eqnarray*}
 with  $ \mu_{l+1,p}(y)\in W_h^y$ is the solution of the following equation: 
\[
    \int_{c}^d \mu_{l+1,p}(y) v(y) dy=\int_{c}^d\bar c_{i,p}v(y)dy,\ \ \forall v\in \mathbb P_{k-2}(y). 
\]
In light of the conclusion in Lemma \ref{lemma:000}, we have     
\begin{eqnarray*}
   \|\partial_y^n\mu_{l+1,p}\|_{0,\infty}&\lesssim& \|\partial_y^n\mu_{l,p-1}\|_{0,\infty}+\|\partial_y^n\mu_{l,p+1}\|_{0,\infty}+
   \|\partial_y^{n+2}\bar\mu_{l-1,p}\|_{0,\infty}\\
   &\lesssim& \|\lambda_1\|_{n+l+1,\infty}+\|\lambda_2\|_{n+l+1,\infty}.
\end{eqnarray*}
 In other words, \eqref{eqq:7} is also valid for $l+1$ and thus holds true for all $l\le k-2$.

  We next prove \eqref{opti:wl}. By \eqref{eqq:7}, we easily get 
\[
    \|\partial_y^n\partial_xw_l\|_{0,\infty}\lesssim h^{l-1} \sum_{p=k-l-1}^k \|\partial_y^n\mu_{l,p}\|_{0,\infty}\lesssim h^l(
    \|\lambda_1\|_{n+l,\infty}+\|\lambda_2\|_{n+l,\infty}).
\]  
  Recalling the definition of $\lambda_i, i=1,2$ and using the estimates for $E^xu$, we have 
 \[
  \|\partial_y^{n+l}\lambda_1\|_{0,\infty}+\|\partial_y^{n+l}\lambda_2\|_{0,\infty}
  \lesssim \|\partial_y^{n+l} E^xu\|_{0,\infty}
  \lesssim  h^{k+1}\|u\|_{k+l+n+1,\infty}.
 \]
   Then 
\[
\|\partial_y^n\partial_xw_l\|_{0,\infty}\lesssim h^{k+l}\|u\|_{k+l+n+1,\infty},\ \ l\in\bZ_{k-2}.
\]
  This finishes the proof of the first inequality of \eqref{opti:wl}.  Similarly,   
  by using  \eqref{wwl}  and the estimates of $\lambda_1,\lambda_2$, we have for all $l\le k-3$, 
\begin{eqnarray*}
    \|\partial_y^nw_l\|_{0,\infty}\lesssim h^l \sum_{p=k-l-1}^k \|\partial_y^n\mu_{l,p}\|_{0,\infty}   &\lesssim & h^{k+1+l}\|u\|_{k+l+n+1,\infty}. 
\end{eqnarray*}
 As for $l=k-2$, we have, from the Poincar\'{e} inequality,
\begin{eqnarray*}
   \|\partial^n_yw_{k-2}\|_{0,\infty}\lesssim \|\partial^n_y\partial_xw_{k-2}\|_{0,\infty}
   \lesssim h^{2k-2}\|u\|_{2k-1+n,\infty}.
\end{eqnarray*}
    Then the second inequality of \eqref{opti:wl} follows. 
  This finishes our proof. $\Box$
\end{proof}

  Now we are ready to construct the correction funciton $w_h^x$. Define 
\begin{equation}\label{whx}
  w^x_h(x,y)=\sum\limits_{l=1}^{k-2}w_l(x,y)
\end{equation}
with $w_l$ defined by \eqref{wwx}-\eqref{wwx:1}. We have the following property for the correction function $w_h^x$. 

\begin{theorem}\label{theo:3}
 Let $w^x_h(x,y)\in V_h$ be defined by \eqref{whx}. 
   Then
 \begin{equation}\label{eq:14}
    w_h^x(a,y)=w_h^x(x,c)=w_h^x(x,d)=0,\ \ w_h^x(b,y)=w_{k-2}(b,y).
 \end{equation}
 Furthermore, if  $u\in W^{2k+1,\infty}(\Omega)$, then
\begin{equation}\label{a:4}
   |a(E^xu+w^x_h,\theta)|\lesssim h^{2k-2}\|u\|_{2k+1,\infty}\|\theta\|_{0},\ \ \forall \theta\in W_h.
\end{equation}
\end{theorem}
\begin{proof}
   First, \eqref{eq:14} follows directly from the conclusions in Lemma \ref{lemma:3}, \eqref{wwx:1} and \eqref{wi}.
   Note that any $\theta\in W_h$ can be decomposed into
\[
   \theta=\theta_0+\theta_1, \ \ \theta_1|_{\tau}\in ({\mathbb P}_{k-2}(x)\setminus\mathbb P_{0}(x)) \times \mathbb P_{k-2}(y),\ \theta_0|_{\tau}\in \mathbb P_{0}(x) \times \mathbb P_{k-2}(y).
\]
  Since $\theta_1\in {\mathcal S}^x_h$ with ${\mathcal S}_h^x$ defined by \eqref{sh},  we have, from \eqref{w0}-\eqref{wwx},
\begin{eqnarray*}
     a(w^x_h,\theta_1)&=&\sum_{l=1}^{k-2}( -\alpha \triangle w_l+{\bf \beta}\cdot\nabla w_l+\gamma w_l ,\theta_1)\\
     &=&((\beta_2\partial_y-\alpha\partial_{yy}+\gamma)(w_{k-2}-w_0+w_{k-3}),\theta_1 )+\beta_1(\partial_xw_{k-2}-\partial_xw_0,\theta_1)\\
     &=& I_{\theta_1}-((\beta_2\partial_y-\alpha\partial_{yy}+\gamma)w_0,\theta_1)-\beta_1(\partial_xw_0,\theta_1),
\end{eqnarray*}
  where $w_0$ is defined in \eqref{w0}, and 
\[
  I_{\theta}=((\beta_2\partial_y-\alpha\partial_{yy}+\gamma)(w_{k-2}+w_{k-3}),\theta )+\beta_1(\partial_xw_{k-2},\theta).
\]
  As for the term $a(w^x_h,\theta_0)$,
  we use the properties of $w_l$ in \eqref{eqq:7} to obtain that
\[
  \partial_{xx}w_n\bot\mathbb P_0(x), n\in\bZ_{k-2},\ \   \partial_xw_m\bot\mathbb P_0(x), m\in\bZ_{k-3},  \ \
  w_l\bot\mathbb P_0(x), l\in\bZ_{k-4}, \ {\rm if}\ k\ge 4,
\]
   and thus, 
\[
   a(w^x_h,\theta_0)=((\beta_2\partial_y-\alpha\partial_{yy}+\gamma)(w_{k-2}+w_{k-3}),\theta_0 )+\beta_1(\partial_xw_{k-2},\theta_0)=I_{\theta_0}.
\]
 Consequently,
\begin{eqnarray*}
   a(w^x_h,\theta)=a(w^x_h,\theta_0+\theta_1)&=&I_{\theta}-((\beta_2\partial_y-\alpha\partial_{yy}+\gamma)w_0,\theta_1)-\beta_1(\partial_xw_0,\theta_1)\\
    &=& I_{\theta}-a(w_0,\theta_1).
\end{eqnarray*}
  On the other hand,  we have from \eqref{EExu:1},
\begin{eqnarray*}
    a(E^xu,\theta)=
    ((\beta_1\partial_x+\beta_2\partial_y-\alpha\partial_{yy}+\gamma)u_0,\theta_1+\theta_0)=a(u_0,\theta_1)+a(u_0,\theta_0),
 \end{eqnarray*}
  and thus
\begin{equation}\label{eq:10}
    a(E^xu+w^x_h,\theta)=I_{\theta}
    +a(u_0,\theta_0)+a(u_0-w_0,\theta_1).
\end{equation}
 By \eqref{opti:wl} and Cauchy-Schwarz inequality, we have
\begin{align}\label{a:3}
   |I_{\theta}|\lesssim \left(\|\partial_xw_{k-2}\|_0+\sum_{n=0}^2\sum_{l=k-3}^{k-2}\|\partial_y^nw_{l}\|_{0}\right)\|\theta\|_0
   \lesssim h^{2k-2}\|u\|_{2k+1,\infty}\|\theta\|_0.
\end{align}
  As for the term $a(u_0,\theta_0)$,  noticing that $\theta_0\in\mathbb P_{0}(x)$,  we have from  \eqref{vip}
\[
   |a(u_0,\theta_0)|=|((\beta_2\partial_y-\alpha\partial_{yy}+\gamma)u_0,\theta_0)|\lesssim h^m\|u\|_{m+2,\infty}\|\theta_0\|_0,\ \ m\le k+1
\]
  for $k=3$.  While for $k\ge 4$, we have $u_0\bot \mathbb P_{k-4}$, which yields
\[
   a(u_0,\theta_0)=0.
\]
  Consequently, 
\begin{equation}\label{a:1}
    |a(u_0,\theta_0)|\lesssim h^{k+1}\|u\|_{k+3,\infty}\|\theta\|_0\lesssim h^{2k-2}\|u\|_{2k,\infty}\|\theta\|_0,\ \ \forall k\ge 3.
\end{equation}
  To estimate the error $u_0-w_0$, we recall the definition of $u_0,w_0$ in \eqref{vip} and \eqref{w0} and then
  use the estimate of $E^xu$ to obtain 
\begin{eqnarray*}
   \|\partial_y^n(u_0-w_0)\|_{0,\infty,\tau_{i,j}} \lesssim \|\partial_y^n(Q_h^yE^xu-E^xu)\|_{0,\infty}
   \lesssim  h^{2k+1-n}\|u\|_{2k+1,\infty}.
\end{eqnarray*}
  Similarly,  we get
\begin{eqnarray*}
   \|\partial_x(u_0-w_0)\|_{0,\infty,\tau_{i,j}} \lesssim h^{-1}\|Q_h^yE^xu-E^xu\|_{0,\infty}    \lesssim  h^{2k}\|u\|_{2k+1,\infty}.
\end{eqnarray*}
 Consequently,
\[
   |((\beta_2\partial_y-\alpha\partial_{yy}+\gamma)(u_0-w_0),\theta_1)+\beta_1(\partial_x(u_0-w_0),\theta_1)|\lesssim h^{2k-2}\|u\|_{2k+1,\infty}\|\theta_1\|_0.
\]
 Substituting \eqref{a:3}-\eqref{a:1} into \eqref{eq:10} yields
\begin{eqnarray*}
|a(E^xu+w^x_h,\theta)|
&\lesssim &  h^{2k-2}\|u\|_{2k+1,\infty}(\|\theta\|_0+\|\theta_1\|_0)
\lesssim h^{2k-2}\|u\|_{2k+1,\infty}\|\theta\|_0.
\end{eqnarray*}
  This finishes the proof of \eqref{a:4}.   $\Box$

\end{proof}

\subsection{Correction function for the error  $a(E^yu,\theta)$}

  The construction of the correction function $w_h^y$ for $a(E^yu,\theta)$ is similar to that of $w_h^x$. 
  To be more precise, we suppose 
  in each element $\tau_{ij}$, 
\[
   E^yu|_{\tau_{ij}}=\sum_{q=k+1}^{\infty}\varsigma_{j,q}(x)J_{j,q}^{-2,-2}(y).
\]
  Let 
 \[
    \bar w_0(x,y)|_{\tau_{ij}}=Q_h^x\varsigma_{j,k+1}(x)J_{i,k+1}^{-2,-2}(y)+Q_h^x\varsigma_{j,k+2}(x) J_{i,k+2}^{-2,-2}(y),\ \ \bar w_{-1}(x,y)=0.
 \]
  For all $l\in\bZ_{k-2}$,  we define a series of functions $\bar w_l\in V_h$ as follows:
  \begin{eqnarray*}\label{ww}
  &&\alpha(\partial_{yy}\bar w_{l},\theta)=(\beta_2\partial_y \bar w_{l-1}+ (\beta_1\partial_x-\alpha\partial_{xx}+\gamma)\bar w_{l-2},\theta ),\ \ \forall\theta\in{\mathcal S}^y_h,\\ \label{wwy:1}
  &&\partial_{y}\bar w_{l}(x,y_j)=0,\ \ \bar w_{l}(x,c)=0, \bar w_{l}(a,y)=\bar w_{l}(b,y)=0,\ \ \forall (x,y)\in\Omega, 
\end{eqnarray*}
  where
\[
   {\mathcal S}^y_h:=\{\theta(x,y): \theta\in  \mathbb P_k(x)\times(\mathbb P_{k-2}(y)\setminus \mathbb P_0(y))\}.
\]
  Following the same argument as that in Lemma \ref{lemma:3}, we get $\bar w_l(x,y_j)=0$ for all $l\le k-3$ and
\begin{align}\label{opti:w2}
  \|\partial_x^n \partial_y\bar w_l\|_{0,\infty}\lesssim h^{k+l}\|u\|_{k+l+n+1,\infty},\
    \|\partial_x^n \bar w_l\|_{0,\infty}\lesssim h^{\min(k+l+1,2k-2)}\|u\|_{k+l+n+1,\infty}.
\end{align}
  Define
\begin{equation}\label{wy}
  w_h^y(x,y)=\sum_{l=1}^{k-2}\bar w_l(x,y),
\end{equation}
  and follow what we have done  in Theorem \ref{theo:3}, we get
\begin{equation}\label{eq:15}
    w_h^y(x,c)=w_h^y(a,y)=w_h^y(b,y)=0,\ \ w_h^y(x,d)=\bar w_{k-2}(x,d),
\end{equation}
  and
\begin{equation}\label{eq:13}
   |a(E^yu+w^y_h,\theta)|\lesssim h^{2k-2}\|u\|_{2k+1,\infty}\|\theta\|_{0},\ \ \forall \theta\in W_h.
\end{equation}

\subsection{Proof of Proposition \ref{proposition:2}}

Define the correction function by
\[
  w_h(x,y)=(w_h^x+w_h^y)(x,y)-\frac{x-a}{b-a}w_{k-2}(b,y)-\frac{y-c}{d-c}\bar w_{k-2}(x,d),
\]
  where $w_h^x,w_h^y$ are given by \eqref{whx}, \eqref{wwx}-\eqref{wwx:1} and \eqref{wy}.
As a direct consequence of \eqref{eq:14} and \eqref{eq:15}, we have
\[
   w_h(a,y)=w_h(b,y)=w_h(x,c)=w_h(x,d)=0.
\]
  In other words, $w_h\in V_h^0$.

  Now we are ready to prove the conclusion of Proposition \ref{proposition:2}.

\begin{proof}  By using the properties and estimates  of $E^xE^yu$ in Lemma \ref{lemma:1}, we have
\[
   |a(E^xE^yu,\theta)|
   \lesssim  h^{2k-2}\|u\|_{2k-1}\|\theta\|_0,\ \ \forall \theta\in W_h.
\]
  Let $$\tilde w(x,y)=\frac{x-a}{b-a}w_{k-2}(b,y)-\frac{y-c}{d-c}\bar w_{k-2}(x,d).$$ By \eqref{opti:wl} and \eqref{opti:w2}, we have
\[
   a(\tilde w,\theta)\lesssim \sum_{n=0}^2(\|\partial_y^nw_{k-2}\|_{0,\infty}+\|\partial_x^n\bar w_{k-2}\|_{0,\infty})\|\theta\|_0
   \lesssim h^{2k-2}\|u\|_{2k+1,\infty}\|\theta\|_0,
\]
  which yields, together with \eqref{a:4} and \eqref{eq:13},
\begin{eqnarray*}
  |a(\eta+w_h,\theta)|
               \lesssim  h^{2k-2}\|u\|_{2k+1,\infty}\|\theta\|_{0}.
\end{eqnarray*}
 Then \eqref{eqq:8} follows.

 We next prove \eqref{esti:ww}-\eqref{esti:ww1}.  By \eqref{opti:wl} and \eqref{opti:w2}, we have
\begin{eqnarray*}
   &&\|w_h\|_{0,\infty}\lesssim \sum_{l=1}^{k-2}(\|w_l\|_{0,\infty}+\|\bar w_l\|_{0,\infty})\lesssim h^{\min{(k+2,2k-2)}}\|u\|_{2k+1,\infty},\\
   &&\|w_h\|_{m,\infty}\lesssim \sum_{l=1}^{k-2}(\|w_l\|_{m,\infty}+\|\bar w_l\|_{m,\infty})\lesssim h^{k+2-m}\|u\|_{2k+1,\infty},\ \ m=1,2. 
\end{eqnarray*}
   Then \eqref{esti:ww} follows.  By \eqref{eqq:7} and \eqref{wi} , we have $\partial_xw_h^x(x_i,y_j)=0$ and
\[
   |\partial_y^n w_h^x(x_i,y_j)|= |\partial_y^n  w_{k-2}(x_i,y_j)|\le \|\partial_y^n w_{k-2}\|_{0,\infty}\lesssim h^{2k-2}\|u\|_{2k+1,\infty},\ n=0,1.
\]
  Similarly,  there holds
\[
   |\partial_x^n w_h^y(x_i,y_j)|= |\partial_y^n  \bar w_{k-2}(x_i,y_j)|\lesssim h^{2k-2}\|u\|_{2k+1,\infty},\ n=0,1,\ \ \partial_yw_h^y(x_i,y_j)=0.
\]
 Consequently,
\begin{eqnarray*}
   |w_h(x_i,y_j)|+|\nabla w_h(x_i,y_j)|\lesssim h^{2k-2}\|u\|_{2k+1,\infty}.
\end{eqnarray*}
 This finishes the proof of \eqref{esti:ww1}. The proof is complete. $\Box$

\end{proof}

\end{document}